\documentclass[graybox]{svmult}

\usepackage{graphicx,comment}
\usepackage{epsfig}
\usepackage{latexsym}
\usepackage{array}
\usepackage{amssymb}
\usepackage{amsmath}
\usepackage{amsbsy}
\usepackage{stmaryrd}

\newcommand{\jump}[1]{\llbracket#1\rrbracket}
\usepackage{mathptmx}       
\usepackage{helvet}         
\usepackage{courier}        
\usepackage{type1cm}        
\usepackage{makeidx}         
\usepackage{graphicx}        
\usepackage{multicol}        
\usepackage[bottom]{footmisc}

\makeindex             

\begin{document}

\title*{Stabilised finite element methods for ill-posed problems with
  conditional stability}
\author{Erik Burman}
\institute{Erik Burman \at Department of Mathematics, University College London, London, 
UK--WC1E  6BT, 
United Kingdom, \email{e.burman@ucl.ac.uk}}

%
%
\maketitle

\abstract*{In this paper we discuss the adjoint stabilised finite
element method introduced in \emph{E. Burman, Stabilized finite element methods for nonsymmetric,
  noncoercive and ill-posed problems. Part I: elliptic equations, SIAM
  Journal on Scientific Computing} \cite{Bu13} and how it may be used for the
computation of solutions to problems for which the standard stability theory given
by the Lax-Milgram Lemma or the Babuska-Brezzi Theorem fails. We pay
particular attention to ill-posed problems that have some conditional
stability property and prove (conditional) error estimates in an
abstract framework. As a model problem we consider the elliptic Cauchy
problem and provide a complete numerical analysis for this case.
Some numerical examples are given to illustrate the theory.}

\abstract{In this paper we discuss the adjoint stabilised finite
element method introduced in \emph{E. Burman, Stabilized finite element methods for nonsymmetric,
  noncoercive and ill-posed problems. Part I: elliptic equations, SIAM
  Journal on Scientific Computing}  \cite{Bu13} and how it may be used for the
computation of solutions to problems for which the standard stability theory given
by the Lax-Milgram Lemma or the Babuska-Brezzi Theorem fails. We pay
particular attention to ill-posed problems that have some conditional
stability property and prove (conditional) error estimates in an
abstract framework. As a model problem we consider the elliptic Cauchy
problem and provide a complete numerical analysis for this case.
Some numerical examples are given to illustrate the theory.}
\section{Introduction}
Most methods in numerical analysis are designed making explicit use
of the well-posedness \cite{Hada02} of the underlying continuous problem. This is
natural as long as the problem at hand indeed is well-posed, but even
for well-posed continuous problems the resulting discrete problem may be
unstable if the finite element spaces are not well chosen or if the
mesh-size is not small enough. This is for instance the case for indefinite
problems, such as the Helmholtz problem, or constrained problems such
as Stokes' equations. For problems that are ill-posed on the
continuous level on the other hand the approach makes less sense
and leads to the need of regularization on the continuous level so
that the ill-posed problem can be approximated by solving a sequence
of well-posed problems. The regularization of the continuous problem
can consist for example of Tikhonov regularization \cite{TA77} or a so-called quasi
reversibility method \cite{LL69}. In both cases the underlying problem is
perturbed and the original solution (if it exists) is recovered only in the limit as
some regularization parameter goes to zero. The disadvantage of this
approach from a numerical analysis perspective is that once the
continuous problem has been perturbed to some order,
the accuracy of the computational method must be made to match that of
the regularization. The strength of the regularization on the other
hand must make the continuous problem stable and damp perturbations
induced by errors in measurement data. This leads to a twofold matching problem
where the regularization introduces a perturbation of first order, essentially excluding the efficient use of many tools from
numerical analysis such as high order methods, adaptivity and stabilisation.
The situation is vaguely reminiscent of that in conservation laws
where in the beginning low order methods inspired by viscosity
solution arguments dominated, to later give way for high resolution
techniques, based on flux limiter finite volume schemes or (weakly) consistent stabilised finite element
methods such as the Galerkin Least Squares methods (GaLS) or
discontinuous Galerkin methods (dG) (see for instance \cite{CJST98}
and references therein). These methods allow for high resolution in
the smooth zone while introducing sufficient viscous stabilisation in zones with
nonlinear phenomena such as shocks or rarefaction waves.

In this paper our aim is to advocate a similar shift towards weakly
consistent stabilisation methods for the computation of ill-posed
problems. The philosophy behind this is to cast the problem in the
form of a constrained optimisation problem, that is first discretized,
leading to a possibly unstable discrete problem. The problem is then
regularized on the discrete level using techniques known from the theory of stabilised
finite element methods. This approach has the following potential advantages
some of which will be explored below:
\begin{itemize}
\item the optimal scaling of the penalty parameter with respect to the mesh
  parameter follows from the error analysis;
\item for ill-posed problems where a conditional stability estimate
  holds, error estimates may be derived that are in a certain sense
  optimal
with respect to the discretization parameters;
\item discretization errors and perturbation errors may be handled in
  the same framework;
\item a posteriori error estimates may be used to drive adaptivity;
\item a range of stabilised finite element methods may be used for the
  regularization of the discrete problem;
\item the theory can be adapted to many different problems.
\end{itemize}

Stabilised finite element methods represent a general technique for the
regularization of the standard Galerkin method in order to improve its
stability properties for instance for advection--diffusion problems at high P\'eclet number or
to achieve inf-sup stability for the pressure-velocity coupling in
the Stokes' system. To achieve optimal order convergence the
stabilisation terms must have some consistency properties, i.e. they
decrease at a sufficiently high rate when applied to the exact solution or to any smooth enough function. Such
stabilising terms appear to have much in common with Tikhonov
regularization in inverse problems, although the connection does not
seem to have been made in general. In the recent papers \cite{Bu13,Bu14a} we
considered stabilised finite element methods for problems where
coercivity fails for the continuous problem and showed that optimal
error estimates can be obtained without, or under very weak, conditions on
the physical parameters and the mesh parameters, also for problems where the standard Galerkin
method may fail. 

In the first part of this series \cite{Bu13} we considered the analysis of
elliptic problems without coercivity using duality arguments. The
second part \cite{Bu14a}  was consecrated to problems for which coercivity fails,
but which satisfy the Babuska-Brezzi Theorem, illustrated by the
transport equation. Finally in the note \cite{Bu14b} we extended the
analysis of \cite{Bu13} to the case of ill-posed problems with some {\emph{
  conditional}} stability property. 

Our aim in the present essay is to review and unify some of these results and
give some further examples of how stabilised methods can be used for
the solution of ill-posed problem. To exemplify the theory we
will restrict the discussion to the case of scalar second order
elliptic problems on the form 
\begin{equation}\label{strong}
\mathcal{L} u = f \quad \mbox{ in }\Omega
\end{equation}
where $\mathcal{L}$ is a linear second order elliptic operator, $u$ is the unknown and $f$ is
some known data and $\Omega$ is some simply connected, open subset of
$\mathbb{R}^d$, $d=2,3$. Observe that the operator $\mathcal{L}$ does not
necessarily have to be on divergence form, although we will only consider
this case here to make the exposition concise (see \cite{WaWa15} for an
analysis of well-posed elliptic problems on nondivergence form).

The discussion below will also be restricted to finite element spaces
that are subsets of $H^1(\Omega)$. For the extension of these results
to a nonconforming finite element method we refer to \cite{Bu14c}.

\subsection{Conditional stability for ill-posed problems}
There is a rich literature on conditional stability estimates for
ill-posed problems. Such estimates often take the form of three
sphere's inequalities or Carleman estimates, we refer the reader to
\cite{ARRV09} and references therein. 

The estimates are conditional, in the sense that they only hold under the
condition that the exact solution exists in some Sobolev space $V$, equipped with scalar product
$(\cdot,\cdot)_V$ and associated norm
$\|\cdot\|_V:=(\cdot,\cdot)_V$. Hereinwe will only consider the case
where  $V \equiv H^1(\Omega)$. 
Then we introduce $V_0 \subset V$ and consider the problem: find $u \in V_0$ such that
\begin{equation}\label{abstractprob}
a(u,w) = l(w), \quad \forall w \in W,
\end{equation}
Observe that $V_0$ and $W$ typically are different
subsets of $H^1(\Omega)$ and we do not
assume that $W$ is a subset or $V_0$ or vice versa. The operators
$a(\cdot,\cdot):V\times V \rightarrow \mathbb{R},\, l(\cdot):W
\rightarrow \mathbb{R}$ denote a bounded bilinear and a bounded linear form
respectively. The form $a(\cdot,\cdot)$ is a weak form of
$\mathcal{L} u$. We let $\|\cdot\|_C$ denote the norm for which the
condition must be satisfied and $\|\cdot \|_S$ denote the norm in
which the stability holds.

We then assume that a stability estimate of the
following form holds: if for some $x \in V_0$, with $\|x\|_C \leq E$
there exist $\varepsilon < 1$ and $r \in W'$ such that
\begin{equation}\label{contdep}
\begin{array}{cc}
\left\{\begin{array}{l}
a(x,v) = (r,v)_{\left<W',W \right>} \; \forall v \in W\\
\|r\|_{W'} \leq \varepsilon 
\end{array}\right.;
& \quad  \mbox{ then } \|x\|_S \leq \Xi_E(\varepsilon),
\end{array}
\end{equation}
where $\Xi_E(\cdot):\mathbb{R}^+ \rightarrow
\mathbb{R}^+$ is a smooth, positive, function, depending on the
problem, $\|\cdot\|_S$ and $E$, with $\lim_{s
  \rightarrow 0ˆ+} \Xi_E(s) =0 $. Depending on the problem different
smallness conditions may be required to hold on 
$\varepsilon$. 

The idea is that the stabilised methods we propose may use the estimate
\eqref{contdep} {\emph{directly}} for the derivation of error estimates, without relying on the Lax-Milgram Lemma or the Babuska-Brezzi
Theorem. Let us first make two observations valid also for well-posed problems. When the
assumptions of the Lax-Milgram's lemma are satisifed \eqref{contdep}
holds unconditionally for the energy norm and $\Xi_E(\varepsilon) = C\varepsilon$, for
some problem dependent constant $C$. If for a given problem the adjoint equation $a(v,z)=j(v)$ admits a 
solution $z \in W$, with $\|z\|_{W} \leq E_j$, for
some linear functional $j \in V'$ then
\begin{equation}\label{adjoint_stability}
|j(x)| = |a(x,z)| = |r(z)| \leq  E_j \|r\|_{W'}
\end{equation}
and we see that for this case the condition of the conditional
stability applies to the adjoint solution.

Herein we will focus on the case of the
elliptic Cauchy problem as presented in \cite{ARRV09}. In this problem both
Dirichlet and Neumann data are given on a part of the boundary,
whereas nothing is known on the complement. We will end this section by detailing
 the conditional stability \eqref{contdep} of the elliptic Cauchy problem. We give the result here
 with reduced technical detail and refer to  \cite{ARRV09} for the
 exact dependencies of the constants on the physical parameters and
 the geometry. 
\subsection{Example: the elliptic Cauchy problem}\label{sec:Cauchy_prob}
The problem that we are interested in takes the form
\begin{equation}
\label{eq:Pb}
\left\{
\begin{array}{rcl}
-\nabla \cdot (\sigma \nabla u) + c u& = &  f, \mbox{ in } \Omega\\
u  &=&  0 \mbox{ on } \Gamma_D\\
\partial_n u &=&\psi \mbox{ on } \Gamma_N
\end{array}
\right. 
\end{equation}
where $\Omega \subset \mathbb{R}^d$, $d=2,3$ is a polyhedral (polygonal)
domain with boundary $\partial \Omega$, $\partial_n u:= n^T \cdot \nabla
u$, (with $n$ the outward pointing normal on $\partial \Omega$),
$\sigma \in \mathbb{R}^{d\times d}$ is a symmetric matrix for which
$\exists \sigma_0 \in \mathbb{R}$, $\sigma_0>0$ such that $ y^T\cdot \sigma y >
\sigma_0$ for all $y \in \mathbb{R}^d$ and $c \in \mathbb{R}$. By
$\Gamma_N,\, \Gamma_D$ we denote polygonal subsets of the
boundary $\partial \Omega$, with union $\Gamma_B:= \Gamma_D \cup
\Gamma_N$ and  that overlap on some set of nonzero $(d-1)$-dimensional
measure, $\Gamma_S:=\Gamma_D \cap \Gamma_N
\neq \emptyset$. We denote the
complement of the Dirichlet boundary $\Gamma_D' := \partial \Omega
\setminus \Gamma_D$, the complement of the Neumann boundary $\Gamma_N' := \partial \Omega
\setminus \Gamma_N$ and the complement of their union
$\Gamma_B':= \partial \Omega \setminus \Gamma_B$. To exclude the
well-posed case, we assume that the
$(d-1)$-dimensional measure of $\Gamma_S$ and $\Gamma_B'$ is non-zero. The practical interest in \eqref{eq:Pb} stems from
engineering problems where the 
boundary condition, or its data, is unknown on $\Gamma_B'$, but additional
measurements $\psi$ of the fluxes are available on a part of the accessible
boundary $\Gamma_S$.  This results in an ill-posed reconstruction
problem, that in practice most likely does not have a solution due to  measurement errors in
the fluxes \cite{Belg07}. However if the underlying physical process is stable, (in
the sense that the problem where boundary data is known is well-posed)
we may assume that it allows for a unique solution in the idealized situation
of unperturbed data. This is the approach
we will take below. To this end we assume that $f \in L^2(\Omega)$, $\psi \in L^2(\Gamma_N)$ and
that a unique $u \in H^s(\Omega)$, $s>\tfrac32$ satisfies \eqref{eq:Pb}.
For the derivation of a weak formulation we introduce the spaces
$
V_0:= \{v \in H^1(\Omega): v\vert_{\Gamma_D} = 0 \}$ and
$W:= \{v \in H^1(\Omega): v\vert_{\Gamma_N'} = 0 \},
$ both equipped with the $H^1$-norm and with dual spaces denoted by $V_0'$
and $W'$.

Using these spaces we obtain a weak formulation: find $u \in V_0$ such that
\begin{equation}\label{abstract_prob}
a(u,w) = l(w)\quad \forall w \in W,
\end{equation}
where
$$a(u,w) = \int_\Omega (\sigma \nabla u) \cdot \nabla w  + c u w ~\mbox{d}x,$$ and $$l(w) := 
\int_\Omega f w ~\mbox{d}x + \int_{\Gamma_N} \psi \, w ~\mbox{d}s.
$$

It is known \cite[Theorems 1.7 and 1.9 with Remark 1.8]{ARRV09} that if there exists a solution $u \in
H^1(\Omega)$, to \eqref{abstract_prob}, a
conditional stability of the form \eqref{contdep}
holds provided $0\leq\varepsilon<1$ and 
\begin{equation}\label{stab_1}
\begin{array}{l}
\mbox{$\|u\|_S:=
\|u\|_{L^2(\omega)}$, $\omega \subset \Omega:\,
\mbox{dist}(\omega, \Gamma_B') =: d_{\omega,\Gamma_{B}'}>0$}\\[3mm]
\mbox{
with $\Xi(\varepsilon) = C(E) \varepsilon^\tau$, $C(E)>0$, $\tau:=
\tau(d_{\omega,\Gamma_{B}'}) \in (0,1),\, E=\|u\|_{L^2(\Omega)}$}
\end{array}
\end{equation}
 and for \begin{equation}\label{stab_2}
\begin{array}{l}
\mbox{$\|u\|_S:= \|u\|_{L^2(\Omega)}$ with $\Xi(\varepsilon) = C_1(E)
(|\log(\varepsilon)| + C_2(E))^{-\tau}$}\\[3mm]
\mbox{ with $C_1(E), C_2(E)>0$, $\tau \in
(0,1), \, E=\|u\|_{H^1(\Omega)}.$}
\end{array}
\end{equation}
How to design
accurate computational methods that can fully exploit the power of
conditional stability estimates for their analysis remains a
challenging problem. Nevertheless the elliptic Cauchy problem is particularly well studied. For pioneering
work using logarithmic estimates we refer to \cite{FM86, RHD99} and
quasi reversibility \cite{KS91}. For
work using regularization and/or energy minimisation see \cite{ABF06,ABB06,CN06, HHKC07,HLT11}.
Recently progress has been made using least squares \cite{DHH13} or quasi
reversibility approaches \cite{Bour05,Bour06,BD10a} inspired by conditional
stability estimates \cite{BD10b}. In this paper we draw on our experiences from \cite{Bu14b,Bu14c},
that appear to be the first works where error estimates for stabilised finite element
methods on unstructured meshes have been derived for this type of
problem.  For simplicity we will only consider the
operator $\mathcal{L}u:=-\Delta u + c u$, with $c\in \mathbb{R}$ for
the discussion below. 
\section{Discretization of the ill-posed problem}
We will here focus on discretizations using finite element spaces, but
the ideas in this section are general and may be applied to any
finite dimensional space.

We consider the setting of section \ref{sec:Cauchy_prob}.
Let $\{\mathcal{T}_h\}_h$ denote a family of quasi uniform, shape
regular simplicial triangulations, $\mathcal{T}_h:=\{K\}$, of $\Omega$,
indexed by the maximum simplex diameter $h$. The set of faces of the
triangulation will be denoted by $\mathcal{F}$ and
$\mathcal{F}_{I}$ denotes the subset of interior faces. The unit normal of
a face of the mesh will be denoted $n$, its orientation is arbitrary
but fixed, except on faces in $\partial \Omega$ where the normal is
chosen to point outwards from $\Omega$.
Now let $X_h^k$ denote the finite element space of continuous, piecewise 
polynomial functions on
$\mathcal{T}_h$, 
$$
X_h^k := \{v_h \in H^1(\Omega): v_h\vert_{K} \in \mathbb{P}_k(K),\quad \forall K
\in \mathcal{T}_h\}.
$$
Here $\mathbb{P}_k(K)$ denotes the space of polynomials of degree
less than or equal to $k$ on a simplex $K$.
Letting $(\cdot,\cdot)_X$ denote the $L^2$-scalar product over $X
\subset \mathbb{R}^d$ and $\left<\cdot,\cdot\right>_X$ that over $X
\subset \mathbb{R}^{d-1}$, with associated $L^2$-norms $\|\cdot\|_X$,
we define the broken scalar products and the associated norms by,
\[
(u_h,v_h)_h := \sum_{K \in \mathcal{T}_h}  (u_h,v_h)_K, \quad \|u_h\|_h := (u_h,u_h)_h^{\frac12},
\]
\[
\left< u_h,v_h \right>_{\mathcal{F}} :=  \sum_{F \in \mathcal{F}}
\left<u_h,v_h\right>_F,\quad \|u_h\|_{\mathcal{F}} :=\left<u_h,u_h\right>_{\mathcal{F}}^{\frac12}.
\]
If we consider finite dimensional subspaces $V_{h} \subset V_0$ and
$W_h \subset W$, for instance in the finite element context we may
take $V_{h}:= X_h^k \cap V_0$ and $W_h := X_h^k \cap W$, the
discrete equivalent of problem \eqref{abstractprob} (with $g=0$) reads: find
$u_h:= \sum_{j=1}^{N_{V_h}} u_j \varphi_j \in V_{h}$ such that
\begin{equation}\label{disc_ill}
a(u_h,\phi_i) = l(\phi_i), \quad i=1,\hdots, N_{W_h}
\end{equation}
where the $\{\varphi_i\}$ and $\{\phi_i \}$ are suitable bases for
$V_h$ and $W_h$ respectively and $N_{V_h}:= \mbox{dim}(V_h)$,  $N_{W_h}:= \mbox{dim}(W_h)$
This formulation may be written as the linear system 
\[
A U = L,
\]
where $A$ is an $N_{W_h} \times N_{V_h}$ matrix,
with coefficients $A_{ij}:=a(\varphi_j,\phi_i)$, $U =
(u_1,\hdots,u_{N_V})^T$ and $L= (l(\phi_1),\hdots,
l(\phi_{N_{W_h}}))^T$. Observe that since we have not assumed $N_{V_h}
= N_{W_h}$ this system may not be square, but even if it is, it may
have zero eigenvalues. This implies
\begin{enumerate}
\item non-uniqueness: there exists $\tilde U \in \mathbb{R}^{N_{V_h}}\setminus\{{\boldmath{0}}\}$
    such that $A \tilde U = 0$;
\item non-existence: there exists $L \in \mathbb{R}^{N_{V_h}}$ such
    that $L \not \in \mbox{Im}(A)$.
\end{enumerate}
These two problems actually appear also when discretizing well-posed
continuous models. Consider the Stokes' equation for incompressible
elasticity, for this problem the well-known challenge is to design a
method for which the pressure variable is stable and the velocity
field discretely divergence free. Indeed the discrete
spaces for pressures and velocities must be well-balanced. Otherwise,
there may be spurious pressure modes in the solution, comparable to
point 1. above, or if the pressure space is too rich the solution may
``lock'', implying that only the zero velocity satisfies the
divergence free constraint, which is comparable to 2. above. Drawing
on the experience of the stabilisation of Stokes' problem this analogy
naturally suggests the following approach to the stabilisation of \eqref{disc_ill}.
\begin{itemize}
\item Consider \eqref{disc_ill} of the form $a(u_h,w_h) = l(w_h)$ as the constraint for a
  minimisation problem;
\item minimise some (weakly) consistent stabilisation together
  with a penalty for the boundary conditions (or other data) under the
  constraint;
\item stabilise the Lagrange multiplier (since discrete inf-sup
  stability fails in general).
\end{itemize}
To this end we introduce the Lagrangian functional:
\begin{equation}\label{Lagrange}
\boxed{
\L(u_h,z_h) := \frac12 s_V(u_h - u,u_h - u) - \frac12 s_W(z_h,z_h)+ a_h(u_h,z_h) - l_h(z_h) }
\end{equation}
where $s_V(u_h - u,u_h - u)$ and $s_W(z_h,z_h)$ represents a
penalty term, imposing measured data through the presence of $s_V(u_h-u,u_h-u)$ and
symmetric, weakly
consistent stabilisations for the primal and adjoint problems
respectively. The forms $a_h(\cdot,\cdot)$ and $l_h(\cdot)$ are discrete
realisations of $a(\cdot,\cdot)$ and $l(\cdot)$, that may account for
the nonconforming case where $V_h \not \subset V$ and $W_h \not \subset W$.

 The discrete method that we propose is given by the
 Euler-Lagrange equations of \eqref{Lagrange}, find $(u_h,z_h) \in V_h \times W_h$ such that
\begin{equation}\label{FEM}
\begin{array}{rcl}
a_h(u_h,w_h) - s_W(z_h,w_h) &=&l_h(w_h) \\[3mm]
a_h(v_h,z_h) + s_V(u_h,v_h) &=& s_V(u,v_h),
\end{array}
\end{equation}
for all $(v_h,w_h) \in V_h \times W_h$. This results in a square
linear system regardless of the dimensions of $V_h$ and $W_h$.
Note the appearance of
$s_V(u,v_h)$ in the right hand side of the second equation of \eqref{FEM}. This means only stabilisations for which
$s_V(u,v_h)$ can be expressed using known data may be used. This
typically is the case for residual based stabilisations, but also
allows for the inclusion of measured data in the computation in a
natural fashion. The stabilising terms in \eqref{FEM} are used both to
include measurements, boundary conditions and regularization. In order
to separate these effects we will sometimes write
\[
s_x(\cdot,\cdot):= s^D_x (\cdot,\cdot) + s^S_x(\cdot,\cdot), \quad x= V,W
\]
where the $s^D$ contribution is associated with assimilation of data
(boundary or measurements) and the $s^S$ contribution is associated
with the stabilising terms. For the Cauchy problem $s^D_V(u,v_h)$
depends on $\psi$ and $s^S_V(u,v_h)$ may depend on $f$ as we shall see below.

Observe that the second equation of \eqref{FEM} is a finite
element discretization of the dual problem associated to the
pde-constraint of \eqref{Lagrange}. Hence, assuming that a unique
solution exists for the given data, the solution to approximate is $z=0$. The discrete
function $z_h$ will most likely not be zero, since it is perturbed by the
stabilisation operator acting on the solution $u_h$, which in general
does not coincide with the stabilisation acting on $u$. The precise
requirements on the forms will be given in the next section together
with the error analysis. 
We also introduce the following compact form
of the formulation \eqref{FEM}, find $(u_h,z_h) \in V_h \times W_h$ such that
\begin{equation}\label{compactFEM}
A_h[(u_h,z_h),(v_h,w_h)] = L_h(v_h,w_h) \mbox{ for all } (v_h,w_h) \in V_h \times W_h,
\end{equation}
where
\begin{equation}\label{global_A}
A_h[(u_h,z_h),(v_h,w_h)] := a_h(u_h,w_h) - s_W(z_h,w_h) + a_h(v_h,z_h) + s_V(u_h,v_h)
\end{equation}
and
\[
L_h(v_h,w_h) := l_h(w_h) + s_V(u,v_h).
\]
We will end this section by giving some examples of the construction
of the discrete forms. To reduce the amount of generic constants we
introduce the notation $ a \lesssim b$ for $a \leq C b$ where $C$
denotes a positive constant independent of the mesh-size $h$.
\subsection{Example: discrete bilinear forms and penalty terms for
  the elliptic Cauchy problem}
For the elliptic Cauchy problem of section \ref{sec:Cauchy_prob} we
define $V^k_h$ and $W^k_h$ to be $X_h^k$ (the
superscript will be dropped for general $k$). Then we
use information on the boundary conditions to design a form
$a_h(\cdot,\cdot)$ that is both forward and adjoint
consistent. A penalty term is also added to enforce the boundary condition.

\begin{equation}\label{disc_a2}
a_h(u_h,v_h):=a(u_h,v_h)- \left<
 \partial_n u_h, v_h \right>_{\Gamma_N'}- \left<
 \partial_n v_h, u_h\right>_{\Gamma_D}
\end{equation}
\begin{equation}\label{penalty_p2}
s^D_V(u_h ,w_h) := \gamma_{D}\left< h^{-1} u_h, w_h
\right>_{\Gamma_D}+ \gamma_{D} \left< h \partial_n u_h , \partial_n w_h
\right>_{\Gamma_N},
\end{equation}
where $\gamma_D \in \mathbb{R}_+$ denotes a penalty parameter that
for simplicity is taken to be the same for all the $s^D(\cdot,\cdot)$
terms, it follows that, if $u=g$ on $\Gamma_D$,
\[
s^D_V(u ,w_h) := \gamma_{D}\left< h^{-1} g, w_h
\right>_{\Gamma_D}+ \gamma_{D} \left< h\, \psi ,\partial_n w_h
\right>_{\Gamma_N}.
\]
The adjoint boundary penalty may then be written
\begin{equation}\label{penalty_a2}
s^D_W(z_h ,v_h) := \gamma_D  \left< h^{-1} z_h, v_h
\right>_{\Gamma_N'}+ \gamma_D \left< h \,\partial_n z_h,\partial_n
v_h
\right>_{\Gamma_D'}.
\end{equation}
We assume that the computational mesh $\mathcal{T}_h$ is such that the
boundary subdomains consist of the union of boundary element faces,
i.e. the boundaries of $\Gamma_D$ and $\Gamma_N$ coincide with
element edges.
Finally we let $l_h(v_h)$ coincide with $l(v_h)$ for
unperturbed data. Observe that there is much more freedom in the choice
of the stabilisation for $z_h$ since the exact solution satisfies
$z=0$. We will first discuss the methods so that they are consistent
also in the case $z\ne 0$, in order to facilitate the connection to a
larger class of control problems. Then we will suggest a stronger
stabilisation for $z_h$.
\subsection{Example: Galerkin Least Squares stabilisation}\label{GaLS}
For the stabilisation term we first consider the classical Galerkin
Least Squares stabilisation. Observe that for the finite element
spaces considered herein, the GaLS stabilisation in the interior of
the elements must be complemented with a jump contribution on the
boundary of the element. If $C^1$-continuous approximation spaces are used this
latter contribution may be dropped. First consider the least squares contribution,
\begin{equation}\label{GaLS_stab}
s^S_V(u_h,v_h) := \gamma_S (h^2 \mathcal{L} u_h,
\mathcal{L} v_h)_h +  \gamma_S \left< h 
  \jump{ \partial_n u_h}, \jump{\partial_n  v_h} \right>_{\mathcal{F}_I}, \,
\gamma_S \in \mathbb{R}_+.
\end{equation}
Here $\jump{\partial_n  v_h}$ denotes the jump of the normal derivative
of $v_h$ over an element face $F$.
It then follows that, considering sufficiently smooth solutions, $u\in H^{s}(\Omega)$, $s>3/2$,
\[
s^S_V(u,v_h) := \gamma_S (h^2 f,
\mathcal{L} v_h)_h.
\]
Similarly we define
\[
s^S_W(z_h,w_h) =  \gamma_S (h^2 \mathcal{L}^*
z_h, \mathcal{L}^* w_h)_h +  \gamma_S \left<
  h \jump{\partial_n  z_h}, \jump{\partial_n  w_h} \right>_{\mathcal{F}_I}.
\]
For symmetric operators $\mathcal{L}$ we see that
$s^S_W(\cdot,\cdot)\equiv s^S_V(\cdot,\cdot)$, however in the presence of
nonsymmetric terms they must be evaluated separately. 
\subsection{Example: Continuous Interior Penalty stabilisation}\label{CIP}
In this case we may choose the two stabilisations to be the same,
$s^S_W(\cdot,\cdot) \equiv s^S_V(\cdot,\cdot)$ and
\begin{equation}\label{CIP_stab}
s^S_V(u_h,v_h) =  \gamma_S \left< h^{3} \jump{
    \Delta u_h}, \jump{ \Delta v_h} \right>_{\mathcal{F}_I} + \gamma_S \left< h \jump{\partial_n  u_h}, \jump{\partial_n  v_h} \right>_{\mathcal{F}_I}.
\end{equation}

\subsection{Example: Stronger adjoint stabilisation}
Observe that since the exact solution satisfies $z=0$ we can also use
the adjoint stabilisation
\begin{equation}\label{eq:nonadjointcons_pen}
s^S_W(z_h,w_h) =  \gamma_S (\nabla z_h,\nabla w_h)_\Omega 
\end{equation}
This simplifies the formulation for non-symmetric problems when the
GaLS method is used and reduces the stencil, but the
resulting formulation is no longer adjoint consistent and optimal
$L^2$-estimates may no longer be proved in the well-posed case (see \cite{Bu13} for a discussion).
In this case the formulation corresponds to a weighted least squares
method. This is easily seen by eliminating $z_h$ from the
formulation \eqref{FEM}. 
\subsection{Penalty parameters}
Above we have introduced the penalty parameters $\gamma_S$ and
$\gamma_D$. The size of these parameters play no essential role for
the discussion below. Indeed the convergence orders for unperturbed
data are obtained only under the assumption that $\gamma_S,
\gamma_D>0$. 
Therefore the explicit dependence of the constants in the estimates
will not be tracked. Only in some key estimates, relating to stability
and preturbed data, will we indicate the dependence on the parameters
in terms of $\gamma_{min}:= \min(\gamma_S,\gamma_D)$ or $\gamma_{max}:=. \max(\gamma_S,\gamma_D)$.

\section{Hypothesis on forms and interpolants}\label{sec:hypo}
To prepare for the error analysis we here introduce assumptions on the
bilinear forms. The key properties that are needed are a discrete
stability estimate, that the form
$a_h(\cdot,\cdot)$ is continuous on a norm that is controlled by the
stabilisation terms and that the finite element residual can be
controlled by the stabilisation terms. To simplify the presentation we
will introduce the space $H^s(\Omega)$, with $s \in \mathbb{R}_+$
which corresponds to smoother functions than those in $V$ for which
$a_h(u,v_h)$ and $s_V(u,v_h)$ always are well defined. This typically allows us to treat the data
part $s_V^D$ and the stabilisation part $s_V^S$ together using strong
consistency. A more detailed analysis separating the two contributions
in $s_V$ and handling the conformity error of $a_h$ for $u\in V$
allows an analysis under weaker regularity assumptions.
\begin{description}
\item[\bf{Consistency}:]
If $u \in V \cap H^s(\Omega)$ is the solution of \eqref{strong}, then the following Galerkin
orthogonality holds
\begin{equation}\label{galortho}
a_h(u_h - u,w_h) - s_W(z_h,w_h) = l_h(w_h) - l(w_h),\quad \mbox{ for
  all } w_h \in W_h.
\end{equation} 
\item[\bf{Stabilisation operators:}]
We consider positive semi-definite,
symmetric stabilisation operators,
$
s_V: V_h \times  V_h \mapsto \mathbb{R},\quad s_W: W_h \times W_h \mapsto \mathbb{R}.
$
We assume that $s_V(u,v_h)$, with $u$ the solution of
\eqref{abstractprob} is explicitly known, it may depend on data from
$l(\cdot)$ or measurements of $u$.
Assume that both $s_V$ and $s_W$ define semi-norms on $H^s(\Omega) +V_h$ and
$H^s(\Omega) +W_h$ respectively,
\begin{equation}\label{semi_S}
|v+v_h|_{s_Z} := s_Z(v+v_h,v+v_h)^{\frac12}, \forall v\in
H^s(\Omega),\, v_h \in Z_h, \mbox{ with } Z = V,W.
\end{equation}
\item[\bf{Discrete stability:}]
There exists a semi-norm, $|(\cdot,\cdot)|_{\mathcal{L}}:(V_h+H^s(\Omega)) \times (W_h +
H^s(\Omega)) \mapsto \mathbb{R}$, such that $|v|_{s_V} + |w|_{s_W}
\lesssim |(v,w)|_{\mathcal{L}}$ for $v,w \in (V_h+H^s(\Omega)) \times (W_h +
H^s(\Omega))$. The semi-norm  $|(\cdot,\cdot)|_{\mathcal{L}}$ satisfies the following
stability. There exists $c_s > 0$ independent of $h$ such that for all $(\nu_h,\zeta_h)
\in V_h \times W_h$ there holds
\begin{equation}\label{disc_stab}
c_s |(\nu_h,\zeta_h)|_{\mathcal{L}} \leq \sup_{(v_h,w_h) \in V_h
    \times W_h} \frac{A_h[(\nu_h,\zeta_h),(v_h,w_h) ]}{|(v_h,w_h)|_{\mathcal{L} }}.
\end{equation}
\item[\bf{Continuity:}]
There exists interpolation operators $i_V: V
\mapsto V_h$ and $i_W: W \mapsto W_h \cap W$ and norms $\|\cdot\|_{*,V}$ and $\|\cdot\|_{*,W}$ defined
on $V+V_h$ and $W$ respectively, such that
\begin{equation}
a_h(v-i_V v, w_h) \lesssim \|v-i_V v\|_{*,V} |(0,w_h)|_{\mathcal{L}},\, \forall v \in
V\cap H^s(\Omega),\, w_h \in W_h \label{cont1}
\end{equation}
and for $u$ solution of \eqref{abstractprob},
\begin{equation}
a(u - u_h, w - i_W w) \lesssim
\|w - i_W w\|_{*,W}~\eta_V(u_h),\, \forall w \in W, \label{cont2}
\end{equation}
where the a posteriori quantity $\eta_V(u_h):V_h \mapsto \mathbb{R}$ satisfies $\eta_V(u_h) \lesssim |(u-u_h,0)|_{\mathcal{L}}$ for sufficiently
smooth $u$.
\item[\bf{Nonconformity:}]
We assume that the following bounds hold
\begin{equation}\label{lhs-conf}
|a_h(u_h,i_W w) - a(u_h,i_W w)| \lesssim \eta_V(u_h)  \|w\|_W,
\end{equation}
and 
\begin{equation}\label{rhs-conf}
|l_h(i_W w) - l(i_W w) | \leq \delta_l(h) \|w\|_W,
\end{equation}
where $\delta_l:\mathbb{R}^+ \mapsto \mathbb{R}^+$, is some continuous
functions such that $\lim_{x \rightarrow 0^+} \delta_l(x) = \delta_0$,
with $\delta_0=0$ for
unperturbed data.

Also assume that there exists an interpolation operator $r_V:H^1(\Omega)+V_h \mapsto V_0\cap V_h$ such
that
\begin{equation}\label{conf_error_rV}
\|r_V u_h - u_h\|_S +  \|r_V u_h - u_h\|_C +  \|r_V u_h - u_h\|_V \lesssim \eta_V(u_h).
\end{equation}
We assume that $r_V$ has optimal approximation properties in the
$V$-norm and the $L^2$-norm for functions in $V_0 \cap H^s(\Omega)$.
\item[\bf{Approximability:}]
 We assume that the interpolants $i_V:V \mapsto V_h$, $i_W:W\mapsto
 W_h \cap W$
have the following approximation and stability properties. For all
$v\in V\cap H^s(\Omega)$
there holds,
\begin{equation}\label{approxint}
| (v-i_V v,0)|_{\mathcal{L}} + \|v - i_V
v\|_{*,V}\leq C_V(v) h^{t}, \mbox{ with } t \ge 1.
\end{equation}
The factor $C_V(v)>0$ will typically depend on some Sobolev norm of
$v$. For $i_W$ we assume that for some $C_W>0$ there holds
\begin{equation}\label{Wstab}
  |i_W w|_{\mathcal{L}} + \|w - i_W w\|_{*,W}\leq   C_W \|w\|_W, \quad \forall w \in W.
\end{equation}
For smoother functions we assume that $i_W$ has approximation
properties similar to \eqref{approxint}.
\end{description}
\subsection{Satisfaction of the assumptions for the methods discussed}\label{satisfaction}
We will now show that the above assumptions are satisfied for the
method \eqref{disc_a2}-\eqref{penalty_p2}
associated to the elliptic Cauchy problem of section \ref{sec:Cauchy_prob}. We will
assume that $u \in V \cap H^s(\Omega)$ with $s>\frac32$. Consider first the
bilinear form given by \eqref{disc_a2}. To prove the Galerkin
othogonality an integration
by parts shows that
\begin{multline*}
a_h(u,w_h) = (\mathcal{L} u,w_h) + \left<\partial_n u, w_h
\right>_{\Gamma_N}= (f,w_h)+\left<\psi, w_h
\right>_{\Gamma_N}  \\
 = l(w_h)-l_h(w_h)+a_h(u_h,w_h) - s_W(z_h,w_h).
\end{multline*}

It is immediate by inspection that the stabilisation operators defined
in sections \ref{GaLS} and \ref{CIP} both define the semi-norm \eqref{semi_S}.
Now define the semi-norm for discrete stability
\begin{multline}\label{Lnorm}
|(u_h,z_h)|_{\mathcal{L}} := \|h
  \mathcal{L} u_h \|_h + \|h
  \mathcal{L}^* z_h \|_h + \|h^{\frac12} \jump{\partial_n
    u_h}\|_{\mathcal{F}_I}+ \|h^{\frac12} \jump{\partial_n
    z_h}\|_{\mathcal{F}_I}\\
+ \|h^{-\frac12} u_h\|_{\Gamma_D} + \|h^{\frac12} \partial_n u_h\|_{\Gamma_N} + \|h^{-\frac12} z_h\|_{\Gamma_N'}+ \|h^{\frac12} \partial_n z_h\|_{\Gamma_D'} .
\end{multline}
If the adjoint stabilisation \eqref{eq:nonadjointcons_pen} is used a
term $\|z_h\|_{H^1(\Omega)}$ may be added to the right hand side of \eqref{Lnorm}.
Observe that for the GaLS method there holds for $c_s \approx
\gamma_{min} > 0$,
\[
c_s |(u_h,z_h)|_{\mathcal{L}}^2 \leq  A_h[(u_h,z_h),(u_h,-z_h)]
\]
which implies \eqref{disc_stab}.
For the CIP-method one may also prove the inf-sup stability
\eqref{disc_stab}, we detail the proof in appendix.

For the continuity \eqref{cont1} of the form $a_h(\cdot,\cdot)$
defined by equation \eqref{disc_a2}, integrate by parts, from the left
factor to the right, with $\phi \in V_h+H^s(\Omega)$ and apply the
Cauchy-Schwarz inequality,
\begin{multline*}
a_h(\phi,w_h) \leq |(\phi,\mathcal{L}^* w_h)_h| + \left<|\phi|,
  |\jump{\partial_n w_h}| \right>_{\mathcal{F}_I} + |\left< \partial_n \phi,
  w_h \right>_{\Gamma_N'}| + |\left<\phi,
  \partial_n w_h \right>_{\Gamma_D'}| \\[3mm]
\leq \left(\|h^{-1} \phi\|_{\Omega}+ \|h^{-\frac12} \phi\|_{\mathcal{F}_I\cup \Gamma_D'}+
 \|h^{\frac12} \partial_n \phi\|_{\Gamma_N'}\right)  |(0,w_h)|_{\mathcal{L}}.
\end{multline*}
From this inequality we identify the norm $\|\cdot\|_{*,V}$ to be
\[
\|\phi\|_{*,V}:=  \|h^{-1} \phi\|_\Omega+ \|h^{-\frac12} \phi\|_{\mathcal{F}_I\cup \Gamma_D'}+
  \|h^{\frac12} \partial_n \phi\|_{\Gamma_N'}.
\]
Similarly to prove \eqref{cont2} for the form \eqref{disc_a2} let
$\varphi \in W$ and integrate by parts in
$a(u-u_h,\varphi)$, identify the functional 
$\eta_V(u_h)$ and apply the Cauchy-Schwarz inequality with
suitable weights,
\begin{multline}\label{contcont}
a(u - u_h,\varphi) = (f,\varphi) + \left<\psi,\varphi \right>_{\Gamma_N}
- a(u_h,\varphi) \\
\leq  |(f- \mathcal{L} u_h,\varphi)_h| + \left<
  |\jump{\partial_n u_h}|,|\varphi| \right>_{\mathcal{F}_I}+ |\left<
  \psi - \partial_n u_h,
  \varphi \right>_{\Gamma_N}|\\
\leq ( \|h^{-1} \varphi\|_{\Omega}+ \|h^{-\frac12}
\varphi\|_{\mathcal{F}_I\cup \Gamma_N}) \,\eta_V(u_h),
\end{multline}
where we define
\[
\eta_V(u_h):=\|h(f- \mathcal{L} u_h)\|_h+\|h^{-\frac12} \jump{\partial_n u_h}\|_{\mathcal{F}_I}+\|h^{\frac12}  (\psi - \partial_n u_h)\|_{\Gamma_N}+ \|h^{-\frac12} u_h\|_{\Gamma_D}
\]
with $\eta_V(u_h) = |(u-u_h,0)|_{\mathcal{L}}$ and we may identify
\[
\|\varphi\|_{*,W}:=  \|h^{-1} \varphi\|_{\Omega}+ \|h^{-\frac12} \varphi\|_{\mathcal{F}_I\cup \Gamma_N}.
\]
It is important to observe that the continuity \eqref{contcont} holds
for the continuous form $a(\phi,\varphi)$, but not for the discrete
counterpart $a_h(\phi,\varphi)$, since it is not well defined for
$\varphi \in W$. 

For the definition of $i_V$ and $i_W$ we may use
Scott-Zhang type interpolators, preserving the boundary conditions on
$V_0$ and $W$, for $r_V$ we use an nodal interpolation operator in the
interior such that $r_V u\vert_{Gamma_D}=0$. For $u\in H^s(\Omega)$ with $s > \tfrac32$ the
approximation estimate \eqref{approxint} then holds with 
\begin{equation}\label{conv_order}
t:=
\min(s-1,k) \mbox{ and } C_V(u) \lesssim
\|u\|_{H^{t+1}(\Omega)}.
\end{equation} 
The bound
\eqref{Wstab} holds by inverse and trace inequalities and the $H^1$-stability of the Scott-Zhang
interpolation operator. It is also known that \[
\|u_h - r_V u_h\|_{H^1(\Omega)} \lesssim \|h^{-\frac12}
u_h\|_{\Gamma_D}\lesssim \eta_V(u_h).
\]
from which \eqref{conf_error_rV} follows. 
The following relation shows \eqref{lhs-conf},
\begin{multline}\label{conf_bound}
|a_h(u_h,i_W w) - a(u_h,i_W w)| = |\underbrace{\left<\partial_n u_h,{i_W
  w} \right>_{\Gamma_N'}}_{=0} + \left<\partial_n i_W
  w, u_h\right>_{\Gamma_D}| \\
\leq \|h^{\frac12} \partial_n i_W
  w\|_{\Gamma_D} \|h^{-\frac12} u_h\|_{\Gamma_D} \lesssim
  \|w\|_{H^1(\Omega)} \eta_V(u_h).
\end{multline}
Where we used that $i_W w\vert_{\Gamma_N'} = 0$, since $i_W w \in W$.

\section{Error analysis using conditional stability}
We will now derive an error analysis using only the continuous
dependence \eqref{contdep}. First we prove that assuming 
smoothness of the exact solution the error converges with the
rate $h^t$ in the stabilisation semi-norms defined in equation
\eqref{semi_S}, provided that there are no perturbations in data. Then we show that the computational error satisfies a perturbation
equation in the form \eqref{abstract_prob}, and that the right hand side of the perturbation equation can be
upper bounded by the stabilisation semi-norm. Our error bounds are
then a consequence of the assumption
\eqref{contdep}. 
\begin{lemma}\label{lem:stab_conv}
Let $u \in V_0 \cap H^s(\Omega) $ be the solution of
\eqref{abstractprob} and $(u_h,z_h)$ the solution of
the formulation \eqref{compactFEM}. Assume that
\eqref{galortho}, \eqref{disc_stab}, \eqref{cont1} and \eqref{approxint} hold. Then
$$
|(u - u_h,z_h)|_{\mathcal{L}} \lesssim C_V(u) (1+c_s^{-1}) h^{t}.
$$
\end{lemma}
\begin{proof}
Let $\xi_h :=
u_h- i_V u$. By the triangle
inequality $$|(u - u_h,z_h)|_{\mathcal{L}} \leq  |(u - i_V
u,0)|_{\mathcal{L}}+ |(\xi_h,z_h)|_{\mathcal{L}}$$ and the
approximability \eqref{approxint} it is enough to study the error in $ |(\xi_h,z_h)|_{\mathcal{L}}$.
By the discrete stability \eqref{disc_stab}
$$
c_s |(\xi_h,z_h)|_{\mathcal{L}}   \leq \sup_{(v_h,w_h) \in V_h
    \times W_h} \frac{A_h[(\xi_h,z_h),(v_h,w_h) ]}{|(v_h,w_h)|_{\mathcal{L} }}.
$$
Using equation \eqref{galortho} we then have
$$
c_s |(\xi_h,z_h)|_{\mathcal{L}}    \leq \sup_{(v_h,w_h) \in V_h
    \times W_h} \frac{l_h(w_h)-l(w_h)+a_h(u-i_V u,w_h) + s_V(u-i_V u,v_h)}{|(v_h,w_h)|_{\mathcal{L} }}.
$$
Under the assumption of unperturbed data and applying the continuity \eqref{cont1} in the third term of the right
hand side and the Cauchy-Schwarz inequality in the
last we have
\[
a_h(u-i_V u,w_h) + s_V(u-i_V u,v_h) \lesssim (\|u-i_V u\|_{V,*} +
\gamma_{max} |(u-i_V u,0)|_{\mathcal{L}}) |(v_h,w_h)|_{\mathcal{L}}
\]
and hence
\[
c_s |(\xi_h,z_h)|_{\mathcal{L}} \lesssim \|u-i_V u\|_{V,*} +
\gamma_{max}|(u-i_V u,0)|_{\mathcal{L}}.
\]
Applying \eqref{approxint} we may deduce
\[
c_s |(\xi_h,z_h)|_{\mathcal{L}}  \lesssim C_V(u) h^t.
\]
\qed
\end{proof}
\begin{theorem}\label{thm:cont_dep}
Let $u\in V_0 \cap H^s(\Omega)$ be the solution of
\eqref{abstractprob} and $(u_h,z_h)$ the solution of the formulation
\eqref{compactFEM} for which \eqref{galortho}-\eqref{Wstab} hold. Assume that the problem \eqref{abstractprob} has the
stability property \eqref{contdep} and that $u$ and $u_h$ satisfy
the condition for stability. Let $c_a$ define a positive constant
depending only on the constants of inequalities \eqref{cont2},
\eqref{lhs-conf}, \eqref{conf_error_rV} and \eqref{Wstab} and define
the a posteriori quantity
\begin{equation} \label{eta}
\eta(u_h,z_h):= \eta_V(u_h) + |z_h|_{s_W}.
\end{equation} 
Then, if $\eta(u_h,z_h)<c_a^{-1}$, there holds
\begin{equation}\label{aposteriori}
\|u - u_h\|_S\lesssim \Xi_E(c_a\eta(u_h,z_h)) + \eta_V(u_h)
\end{equation}
with $\Xi_E$ independent of $h$.

For sufficiently smooth $u$ there holds
\begin{equation}\label{apriori}
\eta(u_h,z_h)
 \lesssim  C_V(u) (1+c_s^{-1}) h^{t}.
\end{equation}
\end{theorem}
\begin{proof}
We will first write the error as one $V_0$-conforming part and one
discrete nonconforming part. 
It then follows that $e:=u-u_h  = \underbrace{u - r_V u_h}_{= \tilde e
\in V_0} + \underbrace{r_V u_h - u_h}_{= e_h \in V_h}$. Observe that
\[
\|u-u_h\|_{S} \leq \|u-r_V u_h\|_{S} + \|u_h - r_V u_h\|_{S} \leq
\|\tilde e\|_{S} + \eta_V(u_h).
\] 
Since both $u$ and $u_h$ satisfy a stability condition it is also
satisfied for $\tilde e$ 
\begin{equation}\label{V_conf_cond_stab}
\|\tilde e\|_C \leq \|u\|_C + \|u_h\|_C + \|e_h\|_C \lesssim  \|u\|_C +
\|u_h\|_C + \eta_V(u_h) \lesssim 2 E + C_V(u)(1+ c_s^{-1}) h^t.
\end{equation}
Here we used the property that $\eta_V(u_h) \lesssim
|(u-u_h,0)_{\mathcal{L}}| \leq  C_V(u)(1+ c_s^{-1}) h^t$, which
follows from Lemma \ref{lem:stab_conv}. 
Now observe that 
\begin{equation}\label{error_rep}
a(\tilde e,w) = a(e,w) - a(e_h,w) = l(w) - a(u_h,w) - a(e_h,w)
\end{equation}
and since the right hand side is independent of $u$ we identify $r \in W'$ such
that $\forall w \in W$,
\begin{equation}\label{pert_right}
(r,w)_{\left<W',W\right>} := l(w) - a(u_h,w) - a(e_h,w).
\end{equation}
It follows that $\tilde e$ satisfies equation \eqref{abstract_prob} with
right hand side $(r,w)_{\left<W',W\right>}$. Hence since $\tilde e$
satisfies the stability condition estimate 
\eqref{contdep} holds for $\tilde e$.
We must then show that $\|r\|_{W'}$ can be made small under mesh
refinement. We proceed using an argument similar to that of Strang's
lemma and \eqref{galortho} to obtain
\begin{multline}\label{pert_strang}
(r,w)_{\left<W',W\right>} =\underbrace{a(u-u_h,w-i_W w)}_{T_1} +  \underbrace{l(i_W w) - l_h(i_W w)}_{T_2}\\ + \underbrace{a_h(u_h,i_W
w) - a(u_h,i_W w)}_{T_3} - \underbrace{a(e_h,w)}_{T_4}-\underbrace{s_W(z_h,i_W w)}_{T_5}.
\end{multline}
We now use the assumptions of section \ref{sec:hypo} to bound the terms $T_1$-$T_5$. First by
\eqref{cont2} and \eqref{Wstab} there holds
\[
T_1 = a(u-u_h,w-i_W w) \lesssim \eta(u_h,0) \|w-i_W
w\|_{*,W} \lesssim \eta(u_h,0) \|w\|_W.
\]
By the assumption of unperturbed data and exact quadrature we have
$T_2=0$. Using the bound of the conformity error \eqref{lhs-conf} we
obtain for $T_3$
\[
T_3 = a_h(u_h,i_W w) - a(u_h,i_W w) \lesssim \eta(u_h,0) \|w\|_W.
\]
For the fourth term we use the continuity of $a(\cdot,\cdot)$, \eqref{Wstab} and the
properties of $r_V$ to write
\[
T_4 = a(e_h,i_W w) \lesssim \|e_h\|_V \|i_W w\|_W \lesssim \eta(u_h,0)\| w\|_W.
\]
Finally we use the Cauchy-Schwarz inequality and the stability of
\eqref{Wstab} to get the bound
\[
T_5 = s_W(z_h,i_W w) \leq |z_h|_{s_W}  |i_W w|_{s_W}  \lesssim  \eta(0,z_h) \| w\|_W.
\]
Collecting the above bounds on $T_1$,...,$T_5$ in a bound for
\eqref{pert_strang} we obtain
\[
|(r,w)_{\left<W',W\right>}| \lesssim \eta(u_h,z_h) \|w\|_W.
\]
We conclude that there exists $c_a>0$ such that
$
\|r\|_{W'} < c_a \eta(u_h,z_h)$. Applying the conditional
stability we obtain the bound
\[
\|\tilde e\|_S \lesssim \Xi_E(c_a \eta(u_h,z_h))
\]
where the constants in $\Xi_E$ are bounded thanks to the assumptions on
$u$ and $u_h$ and \eqref{V_conf_cond_stab}.

The a
posteriori estimate \eqref{aposteriori} follows using the triangle
inequality and \eqref{conf_error_rV},
\begin{equation}
\|u-u_h\|_S = \|\tilde e + e_h\|_S \leq \|\tilde e\|_S +
\|e_h\|_S
\lesssim \Xi_E(\eta(u_h,z_h)) + \eta_V(u_h).
\end{equation}
The upper bound of \eqref{apriori} is then an immediate
consequence of the inequality $$\eta(u_h,z_h) \leq
|(u-u_h,z_h)|_{\mathcal{L}}$$ and Lemma \ref{lem:stab_conv}.
\qed
\end{proof}
\begin{remark}
Observe that if weak consistency is used for the proof of Lemma
\ref{lem:stab_conv} and the data and stabilisation parts of the term
$s_V$ are treated separately, then we may show that the a posteriori part of Theorem \ref{thm:cont_dep}
holds assuming only $u \in V$.
\end{remark}
\subsection{Application of the theory to the Cauchy problem}
Since the formulation \eqref{compactFEM} with the forms defined by
\eqref{disc_a2}-\eqref{penalty_a2} and the stabilisations
\eqref{GaLS_stab}, \eqref{CIP_stab} or \eqref{eq:nonadjointcons_pen}
satisfies the assumptions of 
Theorem \ref{thm:cont_dep} as shown in section \ref{satisfaction}, in principle the error estimates hold for these methods
when applied to an elliptic Cauchy problem \eqref{eq:Pb} which admits
a unique solution in $V_0 \cap H^s(\Omega)$, $s>\frac32$. The
order $t$ and the constant $C_V(u)$ of the estimates are given by
\eqref{conv_order}. 

However, some important questions are left
unanswered related to the a priori bounds on the discrete solution $u_h,z_h$.
Observe that we assumed that the discrete solution $u_h$ satisfies the
condition for the stability estimate $\|u_h\|_{C} \leq E$. For the
Cauchy problem this means that $\|u_h\|_{H^1(\Omega)} \leq E$
uniformly in $h$. As we shall see below, this bound can be proven only
under additional regularity assumptions on $u$. Nevertheless we can
prove sufficent stability on the discrete problem to ensure that the
matrix is invertible.
We will first show that the
$\mathcal{L}$-semi-norm \eqref{Lnorm} is a norm on $V_h \times W_h$, which immediately implies the existence
of a discrete solution through \eqref{disc_stab}.
\begin{lemma}
Assume that $|(\cdot,\cdot)|_{\mathcal{L}}$ is defined by \eqref{Lnorm}
and the penalty operator \eqref{penalty_p2}. Then
$|(v_h,y_h)|_{\mathcal{L}}$ is a norm on  $V_h \times W_h$.
Moreover for all $h>0$ and all $k \ge 1$ there exists $u_h,z_h \in V_h \times W_h$ solution to \eqref{compactFEM}, with \eqref{disc_a2}-\eqref{penalty_a2} and either
\eqref{GaLS_stab} or \eqref{CIP_stab} as primal and adjoint
stabilisation or \eqref{eq:nonadjointcons_pen} for adjoint stabilisation.
\end{lemma}
\begin{proof}
The proof is a consequence of norm equivalence on discrete
spaces. We know that $|(v_h,0)|_{\mathcal{L}}$ is a semi-norm. To show
that it is actually norm observe that if $|(v_h,0)|_{\mathcal{L}} = 0$
then $v_h \in H^2(\Omega)$, $\mathcal L v_h\vert_{\Omega} = \partial_n
u_h \vert_{\Gamma_N} = u_h \vert_{\Gamma_D} = 0$. It follows that
$v_h \in H^1(\Omega)$ satisfies \eqref{abstract_prob} with zero data. Therefore by
\eqref{stab_2} $v_h = 0$ and we conclude that
$|(v_h,0)|_{\mathcal{L}}$ is a norm. A similar argument yields the
upper bound for $y_h$.
The existence of discrete solution then follows from the inf-sup
condition \eqref{disc_stab}. If we assume that $L_h(v_h,w_h)=0$ we
immediately conclude that $|(u_h,z_h)|_{\mathcal{L}}=0$ 
by which existence and uniqueness of the discrete solution follows.
\qed
\end{proof}
This result also shows that the method has a unique continuation
property. This property in general fails for the standard Galerkin
method \cite{Sno99}.

In the estimate of Theorem \ref{thm:cont_dep} above we have assumed
that both the exact solution $u$ and the computed approximation $u_h$
satisfy the condition for stability, in particular we need $\|u - r_V
u_h\|_{C} \leq E$ . Since $u$
is unknown we have no choice but assuming that it satisfies the
condition and $u_h$ on the other hand is known so the constant $E$ for
$u_h$ or $r_V u_h$ can be checked a posteriori. From
a theoretical point of view it is however interesting to ask if the
stability of $u_h$ can be deduced from the assumptions on $u$ and the
properties of the numerical scheme only. This question in its general
form is open. We will here first give a complete answer in the case of
piecewise affine approximation of the elliptic Cauchy problem and then
make some remarks on the high order case.

\begin{proposition}\label{stab_cond}
Assume that $\|\cdot\|_C$ is bounded by the $H^1$-norm, that $u \in
H^2(\Omega)$ is the solution to \eqref{abstract_prob} and $(u_h,z_h) \in
V_h \times W_h$, with $k=1$, is the solution to \eqref{compactFEM} with the
bilinear forms defined by \eqref{disc_a2}--\eqref{penalty_p2} and
\eqref{CIP_stab}. Then there holds
\begin{equation}\label{uh_apriori}
\|u_h\|_C \lesssim \|u\|_{H^2(\Omega)}.
\end{equation}
\end{proposition}
\begin{proof}
Observe that by a standard Poincar\'e inequality followed by a discrete Poincar\'e inequality for piecewise
constant functions \cite{EGH02} we have
\begin{multline*}
\|u_h\|_{H^1(\Omega)} \leq \|i_V u\|_{H^1(\Omega)} +
\|i_V u - u_h\|_{H^1(\Omega)} \lesssim \|u\|_{H^2(\Omega)} + h^{-1}
|i_V u - u_h|_{s_V} \\\lesssim \|u\|_{H^2(\Omega)} + h^{-1}
|(i_V u - u_h,0)|_{\mathcal{L}} \lesssim  \|u\|_{H^2(\Omega)}.
\end{multline*}
\qed
\end{proof}

A simple way to obtain the conditional stability in the high order case, if the order $t$ is
known is to add a term $(h^{2t} \nabla u_h,\nabla v_h)_\Omega$ to
$s_V(\cdot,\cdot)$.
This term will be weakly consistent to the right order and implies the
estimate
\[
\|u_h\|_{H^1(\Omega)} \lesssim  h^{-t} |u_h|_{s_V} \lesssim \|u\|_{H^{t+1}(\Omega)}.
\]
An experimental value for $t$ can be obtained by studying the
convergence of $|u_h|_{s_V}+|z_h|_{s_W}$ under mesh refinement.
To summarize we present the error estimate that we obtain for the
Cauchy problem \eqref{eq:Pb} when piecewise affine approximation is
used in the following Corollary to Theorem \ref{thm:cont_dep}.
\begin{corollary}\label{cor:convergence}
Let $u\in H^2(\Omega)$ be the solution to the elliptic
Cauchy problem \eqref{eq:Pb} and $u_h,z_h \in V_h \times W_h$ the
solution of \eqref{compactFEM}, with \eqref{disc_a2}-\eqref{penalty_a2} and either
\eqref{CIP_stab} as primal and adjoint
stabilisation or \eqref{eq:nonadjointcons_pen} for adjoint
stabilisation. Then the conclusion of Theorem \ref{thm:cont_dep} holds
with $\|\cdot \|_S:= \|\cdot\|_{\omega}$,
with $\omega \subseteq \Omega$, the function $\Xi_E$ and $E$ given by
\eqref{stab_1} or \eqref{stab_2} and
\[
\eta(u_h,z_h):=
\|h(f -
  \mathcal{L} u_h) \|_h + \|h \jump{\partial_n
    u_h}\|_{\mathcal{F}_I}\\
+ \|h^{-\frac12} u_h\|_{\Gamma_D} + \|h^{\frac12} (\psi - \partial_n
u_h)\|_{\Gamma_N}+ |z_h|_{s_W}.
\]
In particular there holds for $h$ suffiently small,
\begin{equation}\label{localbound}
\|u - u_h\|_{\omega} \lesssim h^\tau \, \mbox{ with } 0<\tau<1 \mbox{
  when } \mbox{dist}(\omega,\Gamma_B')>0
\end{equation}
and 
\begin{equation}\label{globalbound}
\|u - u_h\|_{\Omega} \lesssim (|\log(C_1 h)|+C_2)^{-\tau} \, \mbox{ with } 0<\tau<1.
\end{equation}
\end{corollary}
\begin{proof}
First observe that it was shown in section \ref{sec:hypo} that the proposed
formulation satisfies the assumptions of Theorem
\ref{thm:cont_dep}. It then only remains to show that the stability condition is
uniformly satisfied, but this was shown in Proposition
\ref{stab_cond}. The estimates \eqref{localbound} and
\eqref{globalbound} are then a consequence of \eqref{stab_1},
\eqref{stab_2}, \eqref{conv_order} and \eqref{apriori}. Observe that
by \eqref{apriori}
the smallness condition on $\eta(u_h,z_h)$ will be satisfied for $h$
small enough.
\qed
\end{proof}
\section{The effect of perturbations in data}
We have shown that the proposed stabilised methods can be considered
to have a certain optimality with respect to the conditional
dependence of the ill-posed problem.  In practice however it is
important to consider the case of perturbed data. Then it is now
longer realistic to assume that an exact solution exists. The above
error analysis therefore no longer makes sense. Instead we must include the
size of the perturbations, leading to error estimates that measure the
relative importance of the discretization error and the error in
data. To keep the discussion concise we will present the theory for
the Cauchy problem and give full detail only in the case of
CIP-stabilisation (the extension to GaLS is straightforward by
introducing the perturbations also in the stabilisation $s_V^S$ under
additional regularity assumptions.)
In the CIP case the perturbations can be included in \eqref{compactFEM} by
assuming that 
\begin{equation}\label{pert_rhs}
l_h(w) := (f + \delta f,w)_{\Omega} + \left<\psi+\delta \psi,w \right>_{\Gamma_N}
\end{equation}
where $\delta f$ and $\delta \psi$ denote measurement errors and the
unperturbed case still allows for a unique solution. We obtain for \eqref{rhs-conf},
\begin{equation}\label{pert_size}
|l_h(w_h) - l(w_h)|:= |(\delta f,w_h) + \left<\delta \psi, w_h
\right>_{\Gamma_N}| \lesssim \|\delta l\|_{(H^1(\Omega))'}\|w_h\|_{H^1(\Omega)}.
\end{equation} 
Similarly
the penalty operator $s_V(u,v_h)$ will be perturbed by a $\delta
s(v_h) := \left<h\, \delta \psi, \partial_n v_h
\right>_{\Gamma_N}$, here depending only on $\delta \psi$, but which may
depend also on measurement errors in the Dirichlet data.
We may then write
\begin{equation}\label{pert_L}
L_h(w_h,v_h):= l_h(w_h)+s_V^D(u,v_h) + \delta s^D(v_h).
\end{equation}

Observe that the perturbations must be assumed smooth enough so that
the above terms make sense, i.e. in the case of the Cauchy problem,
$\delta f
\in (H^1(\Omega))'$ and $\delta \psi \in L^2(\Gamma_N)$. It follows
that $\delta
s^D(v_h) \leq h^{\frac12} \|\delta \psi\|_{\Gamma_N} |v_h|_{s_V}$. 

A natural question to ask is
how the approximate solutions of \eqref{compactFEM} behaves in the
asymptotic limit, in the case where no exact solution exists. In this
case we show that a certain norm of the solution must blow up under mesh refinement.
\begin{proposition}
Assume that $l_h \in (H^1(\Omega))'$, but no $u\in V_0$ satisfies the equation 
\begin{equation}\label{illposed}
a(u,w)=l_h(w), \quad \forall w \in W.
\end{equation} Let $(u_h,z_h)$ be the solution of
\eqref{compactFEM} with the stabilisation chosen to be the CIP-method (section
\ref{CIP}). Then if $s^S_V \equiv s^S_W$,
$$\|h^{-\frac12} u_h\|_{\Gamma_D} + \|\nabla
u_h\|_{\Omega} + |z_h|_{s_W} \rightarrow \infty,\mbox{ when } h
\rightarrow 0.
$$
If $s^S_W(\cdot,\cdot)$ is defined by \eqref{eq:nonadjointcons_pen} then
$$ \| \nabla u_h\|_{\Omega} \rightarrow \infty,\mbox{ when } h
\rightarrow 0.
$$
\end{proposition}
\begin{proof}
Assume that there exists $M \in \mathbb{R}$ such that
\[
\|h^{-\frac12} u_h\|_{\Gamma_D} + \|\nabla
u_h\|_{\Omega} + |z_h|_{s_W} < M
\]
for all $h>0$. It then follows by weak compactness that we may extract
a subsequence $\{u_{h}\}$ for which $u_{h}
\rightharpoonup \upsilon \in V$ as $h \rightarrow 0$. We will now show that this function
must be a solution of \eqref{illposed}, leading to a contradiction. Let
$\phi \in C^\infty \cap W$ and consider
\[
a(\upsilon, \phi) = \lim_{h \rightarrow 0} a(u_h,\phi).
\]
For the right hand side we observe that
\begin{multline*}
a(u_h,\phi) = a_h(u_h,\phi - i_W \phi) +
a(u_h,\phi)- a_h(u_h,\phi)\\
+ s_W(z_h,i_W \phi)+ l_h(i_W \phi - \phi) + l_h(\phi).
\end{multline*}
Now we bound the right hand side term by term. First using an argument
similar to that of
\eqref{contcont}, followed by approximation and trace and inverse inequalities, we have
\begin{equation}\label{b1}
 a_h(u_h,\phi - i_W \phi) \lesssim \|\phi - i_W \phi\|_{*,W}
 |(u_h,0)|_{\mathcal{L}} \leq C h \|\phi\|_{H^2(\Omega)} (\|h^{-\frac12} u_h\|_{\Gamma_D} + \|\nabla
u_h\|_{\Omega}).
\end{equation}
Then using an argument similar to \eqref{conf_bound} recalling that
$\phi$ is a smooth function we get the bound
\begin{equation}\label{b2}
a(u_h,\phi)- a_h(u_h,\phi) \leq h^{\frac12} \|\partial_n
\phi\|_{\Gamma_D} \|h^{-\frac12} u_h\|_{\Gamma_D}.
\end{equation}
For the adjoint stabilisation, first assume that it is chosen to be
the CIP stabilisation and add and subtract $\phi$ in the right
slot to get 
\begin{equation}\label{b3}
 s_W(z_h,i_W \phi) \leq \underbrace{s_W(z_h,\phi)}_{=0} +
 s_W(z_h, i_W \phi - \phi) \leq  C |z_h|_{s_W} h \|\phi\|_{H^2(\Omega)}.
\end{equation}
If the form \eqref{eq:nonadjointcons_pen} is used, we first observe
that testing \eqref{compactFEM} with $v_h=u_h, w_h=-z_h$ yields
\begin{multline*}
|u_h|_{s_V}^2+|z_h|_{s_W}^2 = A_h[(u_h,z_h),(u_h,-z_h)] =
l_h(z_h)+\left<h \psi, \partial_n u_h \right>_{\Gamma_N} 
\\
\leq
\|l_h\|_{(H^1(\Omega))'} |z_h|_{s_W} + h^{\frac12} \|\psi\|_{\Gamma_N} |u_h|_{s_V}.
\end{multline*}
It follows that there exists $M>0$ such that
\[
\|h^{-\frac12} u_h\|_{\Gamma_D} + \|h^{\frac12} \partial_n
u_h\|_{\Gamma_N} + \|h^{-\frac12} z_h\|_{\Gamma_N'} +
\|z_h\|_{H^1(\Omega)} \leq M, \quad \forall h>0.
\]
Assuming also that $\|\nabla u_h\|_{\Omega} \leq M$, we may then extract a subsequence $u_{h}
\rightharpoonup \upsilon \in V$ as $h \rightarrow 0$ and $z_h\rightharpoonup \zeta \in W$ as
$h\rightarrow 0$.
Using similar arguments as above we may show that $\exists C>0$ such
that for all $\varphi \in
V \cap C^\infty$ there holds
\[
a(\varphi,z_h) \leq  C h^{\frac12}
\]
implying that $\zeta = 0$, by \eqref{contdep} and \eqref{stab_2}. Therefore $s_W(z_h,i_W \phi)
\rightarrow 0$, for all $\phi \in W$.
Observing finally that $ l_h(i_W \phi - \phi) \lesssim \|l_h\|_{W'} h
\|\phi\|_{H^2(\Omega)}$ we may collect the bounds \eqref{b1} - \eqref{b3}
to conclude that by density
\[
a(\upsilon,\phi) = \lim_{h \rightarrow 0} a(u_h,\phi) = l_h(\phi),
\quad \forall \phi \in W
\]
and hence that $\upsilon$ is a weak solution to \eqref{illposed}. This
contradicts the assumption that the problem has no solution and we
have proved the claim.
\qed
\end{proof}
To derive error bounds for the perturbed problem we assume that the
$W$-norm can be bounded by the $\mathcal{L}$-norm, in the following
fashion, $\forall w_h \in W_h$
\begin{equation}\label{Lstrength}
\frac{\|w_h\|_W}{|(0,w_h)|_{\mathcal{L}}} \lesssim h^{-\kappa}
\end{equation}
for some $\kappa \ge 0$.
We may then prove the following perturbed versions of Lemma
\ref{lem:stab_conv} and Theorem \ref{thm:cont_dep}. 
\begin{lemma}\label{lem:pert_stab_conv}
Assume that the hypothesis of Lemma \ref{lem:stab_conv} are satisfied, with $L_h(\cdot)$ defined by
\eqref{pert_L} and \eqref{pert_rhs}. Also assume that
\eqref{Lstrength} holds for some $\kappa\ge 0$. Then
$$
|(u - u_h,z_h)|_{\mathcal{L}} \lesssim C_V(u)(1+ c_s^{-1}) h^{t} +
c_s^{-1} h^{\frac12}  \|\delta \psi\|_{\Gamma_N}+  c_s^{-1} h^{-\kappa} \|\delta l\|_{(H^1(\Omega))'}.
$$
\end{lemma}
\begin{proof}
We only show how to modify the proof of Lemma \ref{lem:stab_conv} to account for
the perturbed data. Observe that the perturbation appears when we
apply the Galerkin orthogonality:
$$
c_s |(\xi_h,z_h)|_{\mathcal{L}}    \leq \sup_{\{v_h,w_h\} \in V_h \times
W_h} \frac{\delta l(w_h)+a_h(u-i_V u,w_h) + s_V(u-i_V u,v_h) + \delta s^D(v_h)}{|(v_h,w_h)|_{\mathcal{L} }}
$$
here $\delta l(w_h) :=l_h(w_h)-l(w_h)$.
We only need to consider the upper bound of the additional terms
related to the perturbations in the following fashion
$$
\frac{l_h(w_h)-l(w_h)- \delta s^D(v_h)}{|(v_h,w_h)|_{\mathcal{L}}} \lesssim \|\delta
l\|_{(H^1(\Omega))'}
\frac{\|w_h\|_{H^1(\Omega)}}{|(v_h,w_h)|_{\mathcal{L} }}+
h^{\frac12}  \|\delta \psi\|_{\Gamma_N}
\frac{|v_h|_{s_V}}{|(v_h,w_h)|_{\mathcal{L}}}.
$$
The conclusion then follows as in Lemma \ref{lem:stab_conv} and by applying the
assumption \eqref{Lstrength} and the fact that the $\mathcal{L}$
semi-norm controls $|v_h|_{s_V}$.
\qed
\end{proof}
\begin{remark}
In two instances we can give the precise value of the power $\kappa$. First
assume that the adjoint stabilisation is given by equation
\eqref{eq:nonadjointcons_pen} with  $|(0,w_h)|_{\mathcal{L}}$ defined
by \eqref{Lnorm} with the added $\|w_h\|_{H^1(\Omega)}$ term. It then follows that \eqref{Lstrength} holds with
$\kappa=0$. On the other hand if GaLS
stabilisation or CIP stabilisation are used also for the adjoint
variable and piecewise affine spaces are used for the approximation we
know that by a discrete Poincar\'e inequality \cite{EGH02}
\[
\|w_h\|_{H^1(\Omega)} \lesssim h^{-1} |(0,w_h)|_{\mathcal{L}}
\]
and therefore $\kappa=1$ in this case.
\end{remark}

Similarly the perturbations will enter the conditional stability
estimate and limit the accuracy that can be obtained in the
$\|\cdot\|_S$ norm when the result of Theorem \ref{thm:cont_dep} is applied.
\begin{theorem}\label{thm:pert_conv}
Let $u$ be the solution of
\eqref{abstract_prob} and $(u_h,z_h)$ the solution of the formulation
\eqref{compactFEM} with the right hand side given by \eqref{pert_L}. Assume
that the assumptions \eqref{semi_S}-\eqref{approxint} hold, that the problem \eqref{abstract_prob} has the
stability property \eqref{contdep} and that $u$ satisfies
the condition for stability. 
Let 
\[
\eta_\delta(u_h,z_h):= \eta(u_h,z_h)+\|\delta l\|_{W'}
\]
with $\eta(u_h,z_h)$ defined by \eqref{eta}.
Then for 
$
\eta_\delta(u_h,z_h)
$
small enough,  there holds
\begin{equation}\label{aposteriori_pert}
\|u - u_h\|_S\lesssim \Xi_E(\eta_\delta(u_h,z_h)) + \eta_V(u_h)
\end{equation}
with $\Xi_E$ dependent on $u_h$. 
For sufficiently smooth $u$ there holds
\begin{equation}\label{apriori_pert}
\eta_\delta(u_h,z_h)
 \lesssim C_V(u)(1+ c_s^{-1}) h^{t} +
c_s^{-1} h^{\frac12}  \|\delta \psi\|_{\Gamma_N} + (1+ c_s^{-1} h^{-\kappa}) \|\delta l\|_{(H^1(\Omega))'}.
\end{equation}
\end{theorem}
\begin{proof}
The difference due to the perturbed data appears in the Strang type
argument.
We only need to study the term $T_{2}$ of the equation \eqref{pert_strang} under the
assumption \eqref{pert_size}. Using the $H^1$-stability of the
interpolant $i_W$ we immediately get
\[
T_2 = l(i_W  w) - l_h(i_W w) \lesssim \|\delta l\|_{W'} \|i_W w\|_W
\lesssim \|\delta l\|_{W'} \|w\|_W.
\]
It then follows that 
\[
|(r,w)_{\left<W',W\right>}| \leq (c_a \eta(u_h,z_h) +  C_W \|\delta
l\|_{W'})\|w\|_W \leq c_{\delta,a} \eta_\delta(u_h,z_h) \|w\|_W
\]
and assuming that $c_{\delta,a} \eta_\delta(u_h,z_h)  < 1$, the a posteriori bound follows by applying the conditional stability
\eqref{contdep}. For the a priori estimate we apply the result of
Lemma \ref{lem:pert_stab_conv} and $\|\delta l\|_{W'} \leq \|\delta
l\|_{(H^1(\Omega))'}$.
\qed
\end{proof}
Observe that the function $\Xi_E$ in the error estimate depends on $\|u_h\|_C$
and therefore is not robust. A natural question is how small we can
choose $h$ compared to the size of the perturbations before the
computational error stagnates or even grows. This leads to a delicate
balancing problem since the mesh size
must be small so that the residual is small enough, but not too small,
since this will make the perturbation terms dominate. Therefore the
best we can hope for is a window $0<h_{min}<h<h_{max}$, within which the estimates \eqref{aposteriori_pert} and \eqref{apriori_pert} hold.
We will explore this below for the approximation of the Cauchy problem using piecewise
affine elements.
\begin{corollary}
Assume that the hypothesis of Lemma \ref{lem:pert_stab_conv} and
Theorem \ref{thm:pert_conv} are satisfied. Also assume that there
exists $h_{min}>0$ and $C_\delta(u)>0$ such that
\begin{equation}\label{assump_pert_1}
h_{min}^{-\kappa} \|\delta
l\|_{(H^1(\Omega))'} + h^{\frac12} \|\delta \psi\|_{\Gamma_N}
\leq C_\delta(u) h\quad \mbox{ for } h > h_{min}
\end{equation}
and $h_{max}>0$ so such that for $h_{min}<h < h_{max}$
there holds $\eta_\delta(u_h,z_h)<c_{\delta,a}^{-1}$.
Then for $h_{min}<h <h_{max}$ there exists $\Xi_E(\cdot)$, independent of $u_h$ such that
\eqref{aposteriori_pert} and \eqref{apriori_pert} hold. 
\end{corollary}
\begin{proof}
First observe that by Lemma \ref{lem:pert_stab_conv} and under the assumption \eqref{assump_pert_1} there holds for $h>h_{min}$
$$
|(i_V u - u_h,z_h)|_{\mathcal{L}} \lesssim (C_V(u)+C_\delta(u)) c_s^{-1} h
$$
It follows by this bound and the discrete Poincar\'e inequality, that $\|
u_h\|_{H^1(\Omega)} \lesssim  (C_V(u)+C_\delta(u)) c_s^{-1}$ for
$h>h_{min}$. We may conclude that the condition for stability is satisfied
for $u$, $u_h$ and the discrete error $r_V u_h - u_h$. Therefore,
since the smallness assumptionon $\eta_\delta(u_h,z_h)a$ is satisfied for $h<h_{max}$,
there exists $\Xi_E$ independent of $u_h$ such that estimates \eqref{aposteriori_pert} and \eqref{apriori_pert} hold
when $h_{min}<h <h_{max}$.
\qed
\end{proof}
\section{Numerical examples}
Here we will recall some numerical examples from \cite{Bu13} and
discuss them in the light of the above analysis. We choose $\Omega =
(0,1)\times (0,1)$ and limit the study to
CIP-stabilisation and the case where the primal and adjoint
stabilisations are the same. First we
will consider the case of a well-posed but non-coercive
convection--diffusion equation, $\mathcal{L}:= -\mu \Delta u + \beta
\cdot \nabla u.$ Then we study the elliptic Cauchy
problem with $\mathcal{L} u:=-\Delta u$ for unperturbed and perturbed
data and finally we revisit the
convection-diffusion equation in the framework of the elliptic Cauchy
problem and study the effect of the flow characteristics on the
stability. All computations were carried out on unstructured meshes.
In the convergence plots below the curves have the following
characteristics
\begin{itemize}
\item piecewise affine approximation: square markers;
\item piecewise quadratic approximation: circle markers;
\item full line: the stabilisation semi-norm $|u_h|_{S_V}+ |z_h|_{s_W}$;
\item dashed line: the global $L^2$-norm;
\item dotted line with markers: the local $L^2$-norm.
\item dotted line without markers: reference slopes.
\end{itemize}
\subsection{Convection--diffusion problem with
 pure Neumann boundary conditions}
We consider an example given in \cite{CD11}. The operator is chosen as
\begin{equation}\label{op:conv_diff}
\mathcal{L}(\cdot):= \nabla \cdot( \mu \nabla (\cdot) + \beta \cdot)
\end{equation}
with the physical parameters $\mu=1$,
\[
\beta:= -100 \left(\begin{array}{c} x+y \\ y-x \end{array} \right)
\]
(see the left plot of Figure \ref{velocities_conv}) and the exact solution is given by
\begin{equation}\label{exact_sol}
u(x,y) = 30x(1-x)y(1-y).
\end{equation}
This function satisfies homogeneous Dirichlet boundary conditions and
has $\|u\|_\Omega = 1$. Note that $\|\beta\|_{L^\infty} = 200$ and $\nabla
\cdot \beta = -200$, making the problem strongly noncoercive with a
medium high P\'eclet number. We solve the problem with (non-homogeneous)
Neumann-boundary conditions $(\mu \nabla u + \beta u)\cdot n = g$ on
$\partial \Omega$. The parameters were set to $\gamma_D=10$ and
$\gamma_S=0.01$ for piecewise affine approximation and $\gamma_S=0.001$
for piecewise quadratic approximation. The average value of the approximate solutions has been
imposed using a Lagrange multiplier. The right hand side is then chosen as $\mathcal{L} u$
and for the (non-homogeneous) Neumann conditions, a suitable
right hand side is introduced to make
the boundary penalty term consistent. 
\begin{figure}
\centering
\includegraphics[width=0.41\textwidth]{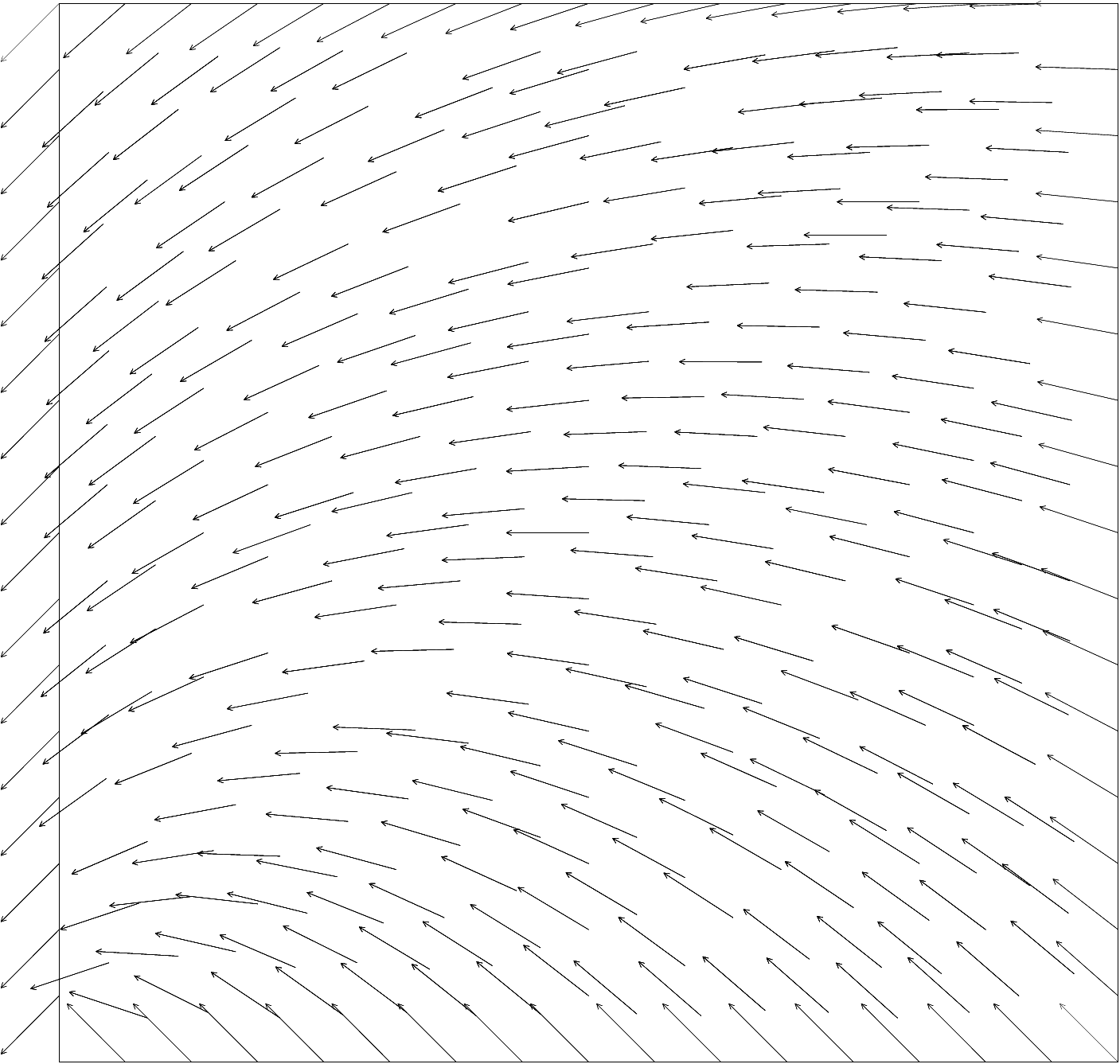}\hspace{0.5cm}
\includegraphics[width=0.45\textwidth, trim = 0cm 2.15cm 0cm 0cm, clip=true]{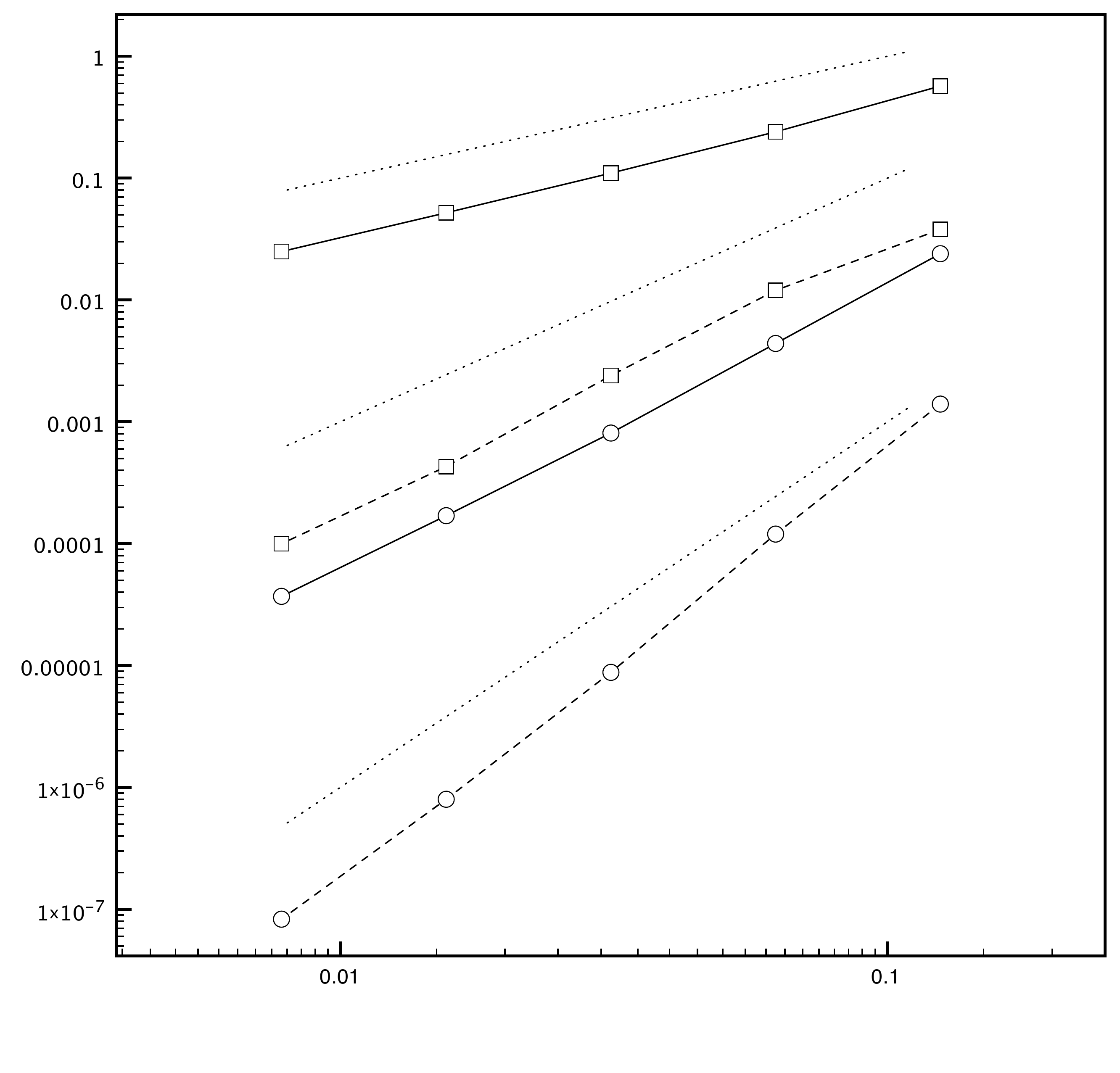} 
\caption{Left: plot of the
  velocity vector field. Right: convergence plot, errors against mesh
  size, filled lines $|u_h|_{s_V}+ |z_h|_{s_W}$, dashed lines
  $L^2$-norm error, dotted lines reference slopes, from top to bottom
  $\mathcal{O}(h)$, $\mathcal{O}(h^2)$, $\mathcal{O}(h^3)$.}\label{velocities_conv}
\end{figure}
In the right plot of Figure \ref{velocities_conv} we observe optimal convergence
rates as predicted by theory (the dual adjoint problem is
well-posed, see \cite{Bu13,CF88}). 
\subsection{The elliptic Cauchy problem}
Here we consider the problem \eqref{eq:Pb} with $\sigma$ the identity
matrix and $c\equiv 0$. We impose the Cauchy data, i.e. both Dirichlet and
Neumann data, on boundaries $x=1,\, 0<y<1$ and $y=1,\, 0<x<1$. We then
solve \eqref{eq:Pb} using the method \eqref{compactFEM} with
\eqref{disc_a2}-\eqref{penalty_a2} and \eqref{CIP_stab} with $k=1$
and $k=2$.

In Figure \ref{penalty_plot}, we present a study of the $L^2$-norm
error under variation of the stabilisation parameter. 
The computations are made on one mesh, with $32$ elements
per side and the Cauchy problem is solved with $k=1,2$ and different
values for $\gamma_{S}$ with $\gamma_{D}=10$
fixed. The  level of $10 \%$ relative error is indicated by the
horizontal dotted line.
Observe that the robustness with
respect to stabilisation parameters is much better for second order polynomial
approximation. Indeed in that case the $10\%$ error level is met for all parameter
values $\gamma_{S} \in [2.0E-5,1]$, whereas in the case of piecewise
affine approximation one has to take $\gamma_{S} \in [0.003,0.05]$.
Similar results for
the boundary penalty parameter not reported here showed that the
method was even more robust under perturbations of $\gamma_{D}$.
In the left plot of Figure \ref{Cauchy_error} we present the contour
plot of the interpolated error $i_V u-u_h$ and in the right, the contour plot of $z_h$. In both cases
the error is concentrated on the boundary where no boundary conditions
are imposed for that particular variable.

In Figure \ref{Cauchy_conv} we present the convergence plots for piecewise
affine and quadratic approximations. The same stabilisation parameters
as in the previous example were used. In both cases we observe the optimal convergence of
the stabilisation terms, $\mathcal{O}(h^k)$, predicted by Lemma \ref{lem:stab_conv}. For
the global
$L^2$-norm of the error we observe experimental convergence of
inverse logarithmic type, as predicted by theory. Note that the main
effect of increasing the polynomial order is a decrease in the error
constant as expected..

For the local $L^2$-norm error,
measured in the subdomain $(0.5,1)^2$, higher convergence
orders, $\mathcal{O}(h^k)$, were obtained in both cases.
\begin{figure}
\centering
\includegraphics[width=8cm]{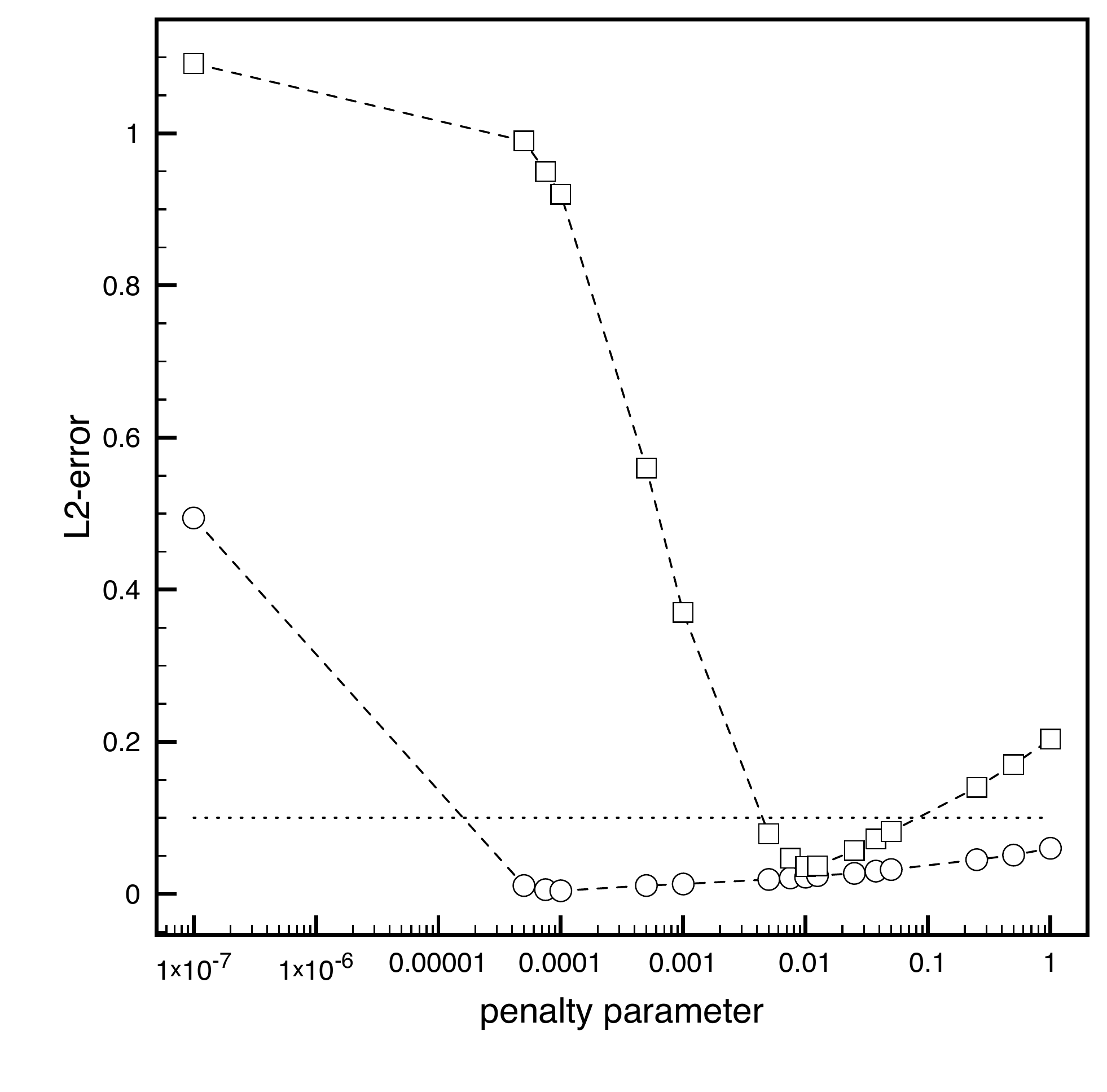}
\caption{Study of the global $L^2$-norm
error under variation of the stabilisation parameter, circles: affine
elements, squares: quadratic elements}\label{penalty_plot}
\end{figure}
\begin{figure}
\centering
\includegraphics[width=0.45\textwidth]{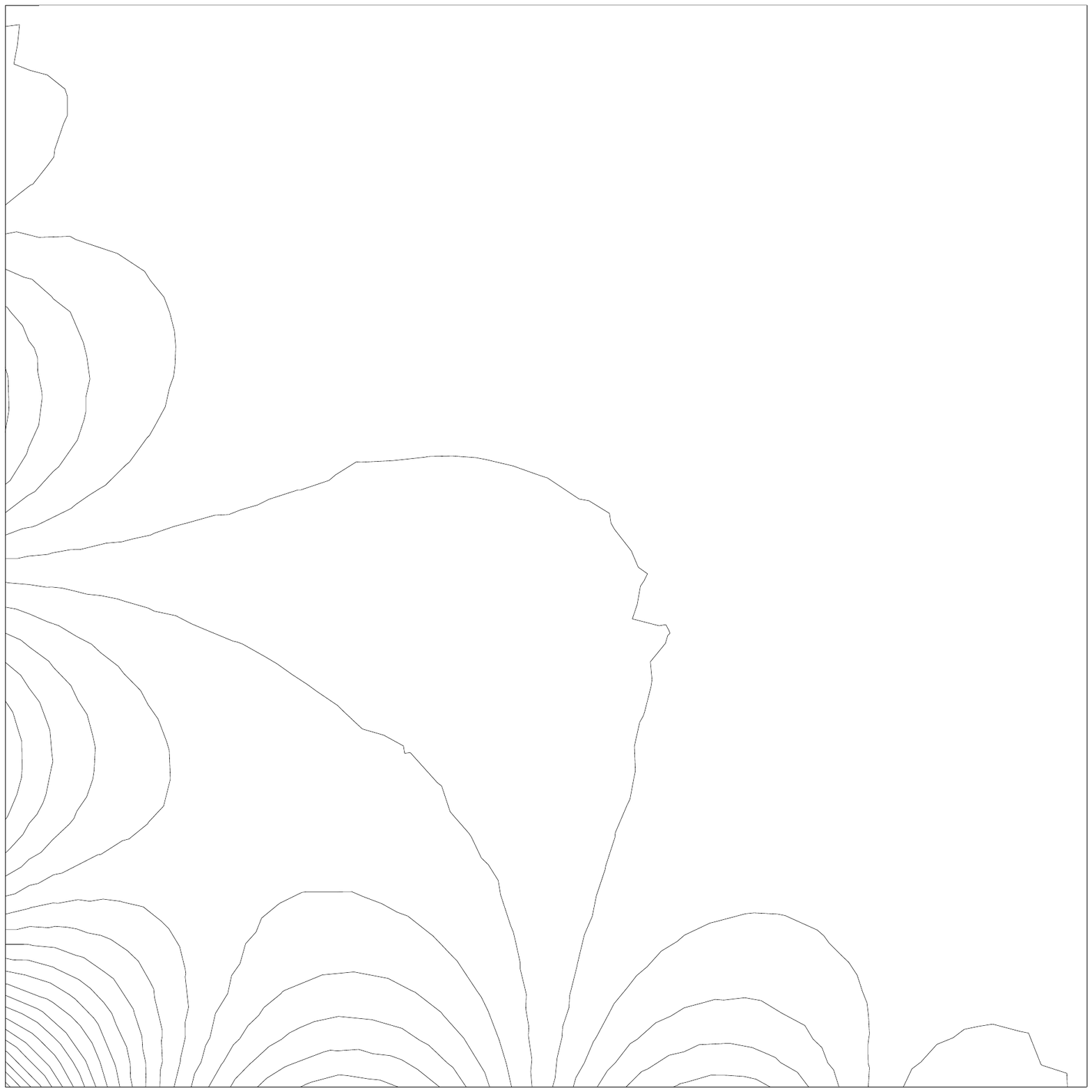}
\includegraphics[width=0.45\textwidth]{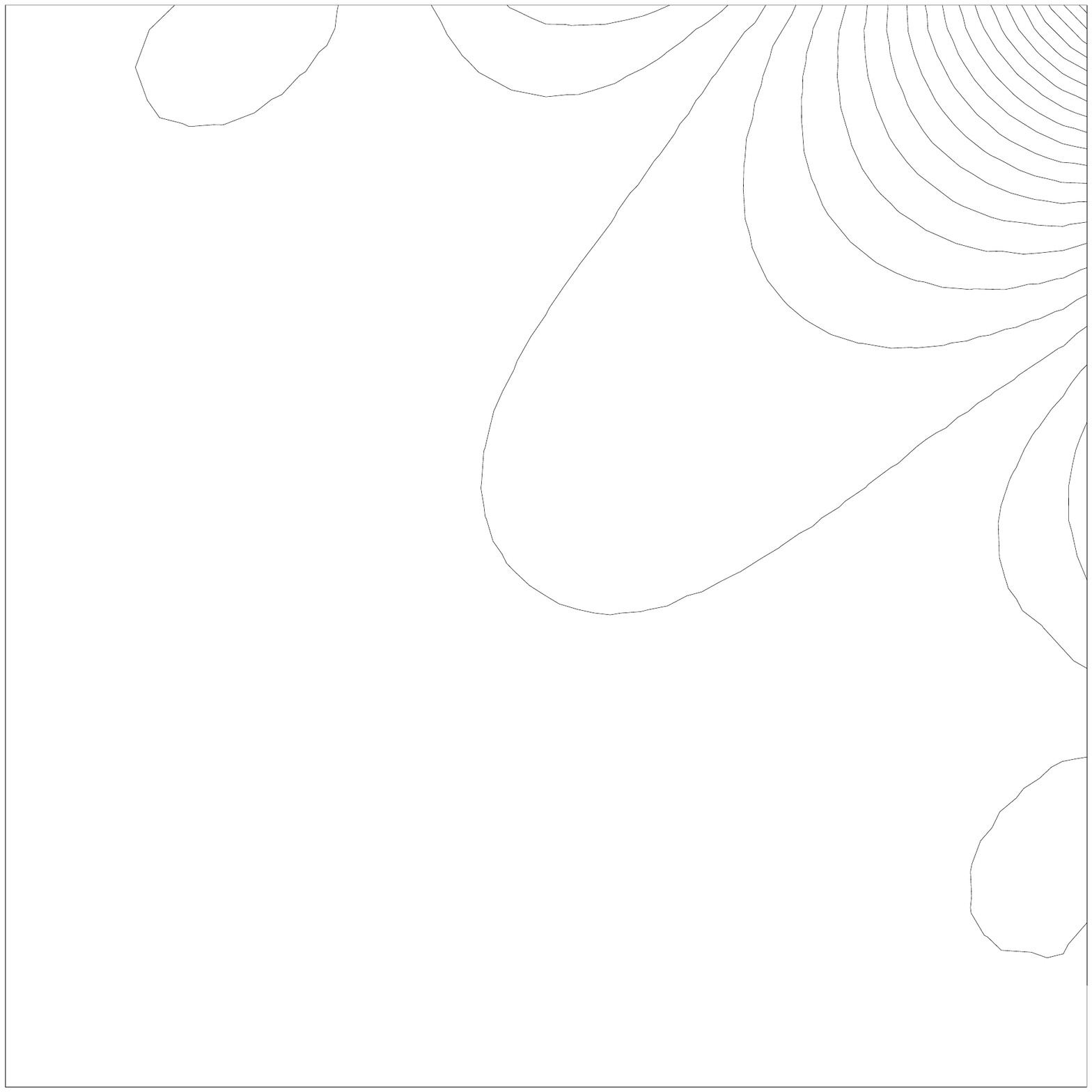}
\caption{Contour plots of the interpolated error $i_h u-u_h$ (left plot) and the
  error in the dual variable $z_h$ (right plot).}\label{Cauchy_error}
\end{figure}
\begin{figure}
\centering
\includegraphics[width=0.47\textwidth]{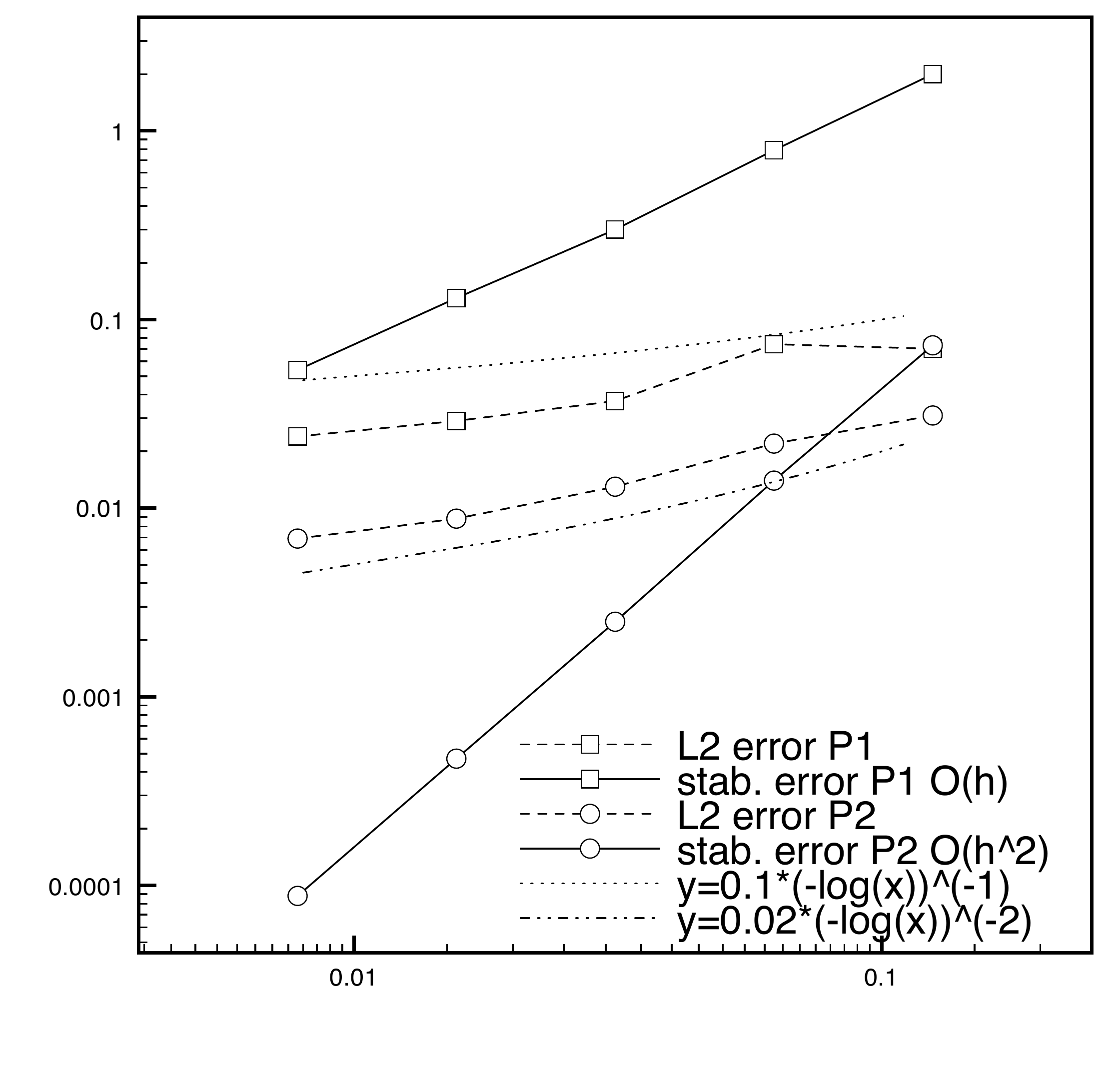}
\includegraphics[width=0.47\textwidth]{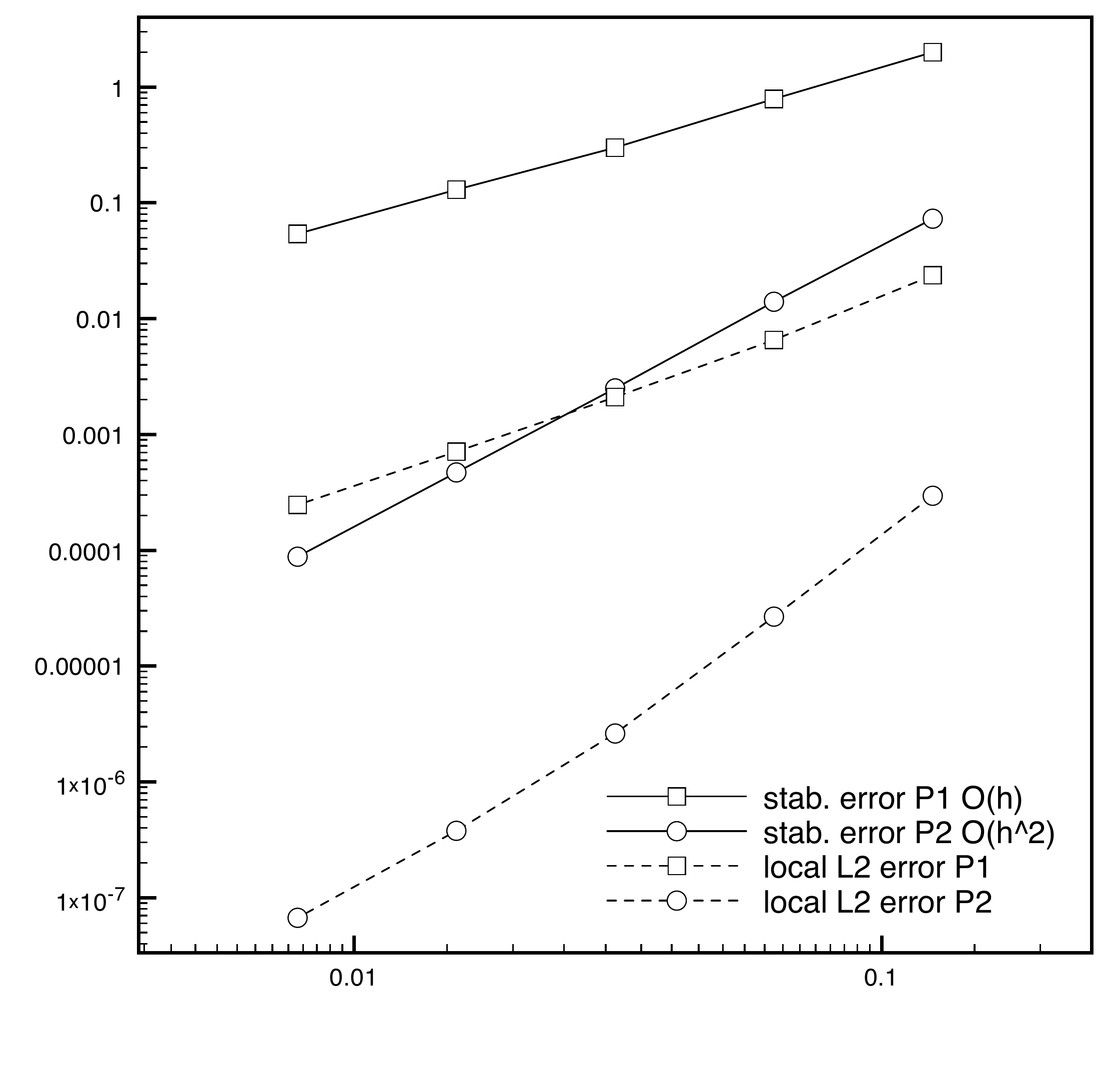}
\caption{Convergence under mesh refinement, the same slopes for the
  stabilization semi-norm are represented in both graphics for reference.}\label{Cauchy_conv}
\end{figure}
\subsubsection{The effect of perturbations in data}
In this section we will consider some numerical experiments with
perturbed data. We consider a
perturbation of the form $\delta \psi =\varsigma v_{rand} \psi$ where
$v_{rand}$ is a random function defined as a fourth order polynomial
on the mesh with random nodal values in $[0,1]$ and $\varsigma>0$ gives
the relative strength of the perturbation. We consider the same
computations as for unperturbed data. 
In all figures we report the stabilisation semi-norm $|z_h|_{s_W} + |u_h|_{s_V}$ to
explore to what extent it can be used as an a posteriori quantity to
tune the stabilisation parameter and to detect loss of convergence due
to perturbed data. 

First we consider the determination of the penalty parameter. First we fix
$\gamma_D=10$. Then, in
Figure \ref{penalty_pert_study} we show the results obtained by varying
$\gamma_S$ when the data is perturbed with $\varsigma=0.01$. We compare the
global $L^2$-error with the stabilisation semi-norm. For the piecewise
affine case we observe that the optimal value of the penalty parameter does
not change much. It is taken in the interval $[0.01,0.1]$, which
corresponds very well with the minimum of the a posteriori quantity $|z_h|_{s_W} +
  |u_h|_{s_V}$. For piecewise quadratic approximation there is a
  stronger difference compared to the unperturbed case. The optimal penalty parameter is now taken in
  the interval $[0.5,5]$. The a posteriori quantity takes its minimum
  value in the interval $[0.1,0.5]$. From this
  study we fix the penalty parameter to $\gamma_S=0.05$ for piecewise
  affine approximation and to $\gamma_S=1.0$ in the piecewise quadratic
  case.

Next we study the sensitivity of the error to variations in the
strength of the perturbation, for the chosen penalty parameters. The
results are given in Figure \ref{sens_study}. As expected the global
$L^2$-error is minimal for the perturbation $\varsigma=0.01$. For
smaller perturbations it remains approximately constant, but for
perturbations larger than $1\%$ the error growth is linear in $\varsigma$ for all
quantities as predicted by theory, assuming the stability condition is satisfied uniformly (see Lemma \ref{lem:pert_stab_conv} and Theorem
\ref{thm:pert_conv}.)

Finally we study the convergence under mesh refinement when $\varsigma=0.01$. The results
are presented in Figure \ref{conv_study}. From the
theory we expect the reduction of the error to stagnate or even start
to grow when $h\lesssim \varsigma$. For the piecewise affine
approximation the minimal global $L^2$-error is $0.065$ for $h=0.015625$
and it follows that the stagnation takes place for $h \approx \varsigma$
in this case. For $k=2$ the minimal global
$L^2$-error is $0.047$ for $h=0.03125$, that is one refinement level
earlier than for the piecewise affine case. In both cases we observe
that the convergence of the stabilisation semi-norm degenerates to
worse than first order immediately after the critical mesh-size. The
dotted lines without markers immediately below the curve representing
the a posteriori quantity are reference curves with slopes
$O(h^{1.1})$ for affine elements and $O(h^{1.4})$ for quadratic
elements $k=2$. This rate is suboptimal in the latter case, indicating
a higher sensibility to perturbations for higher
order approximations. It follows that regardless of the smoothness of the
(unperturbed) exact solution, high order approximation only pays if
perturbations in data are small enough so that they do not dominate
before the asymptotic range is reached.
\begin{figure}
\centering
\includegraphics[width=0.47\textwidth]{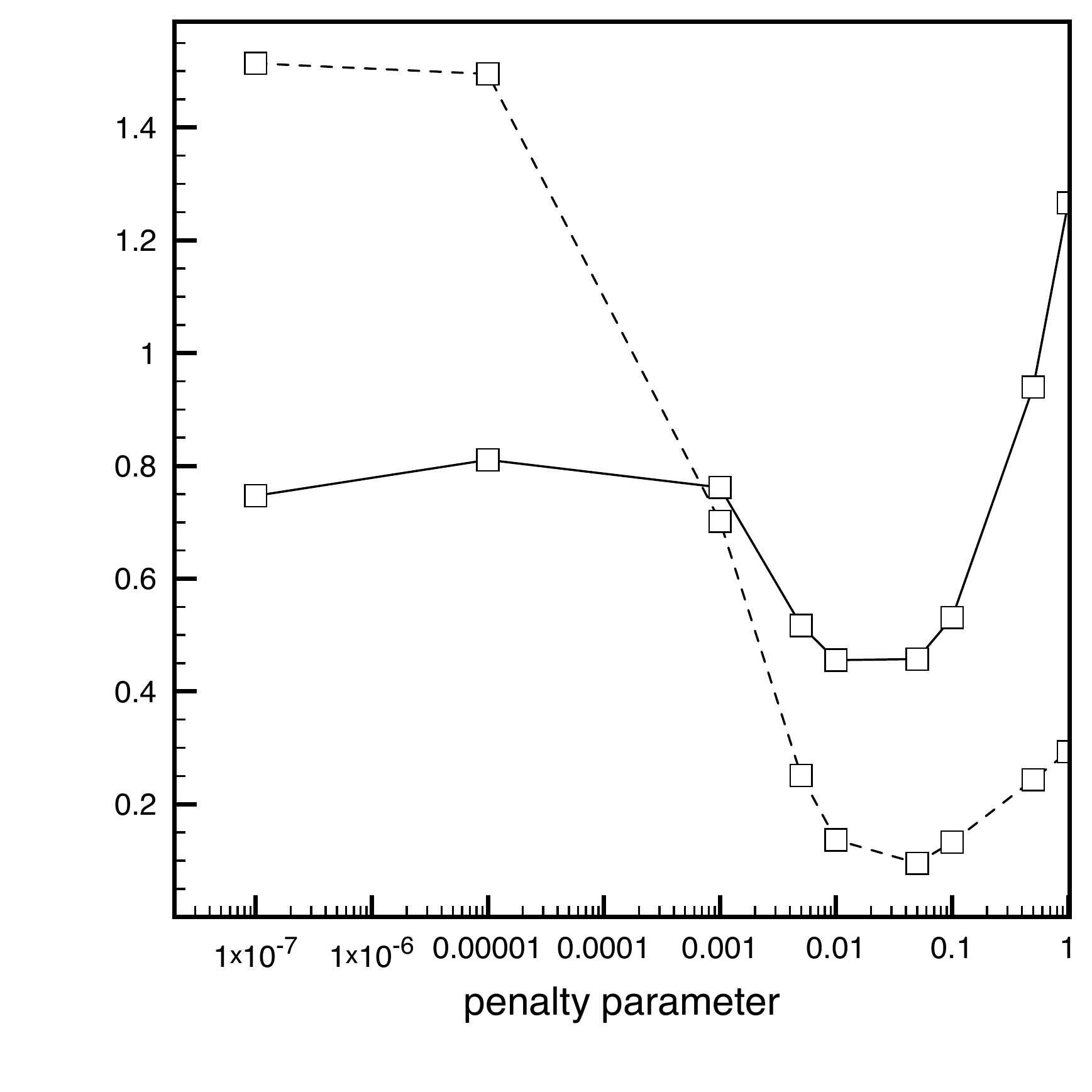}
\includegraphics[width=0.47\textwidth]{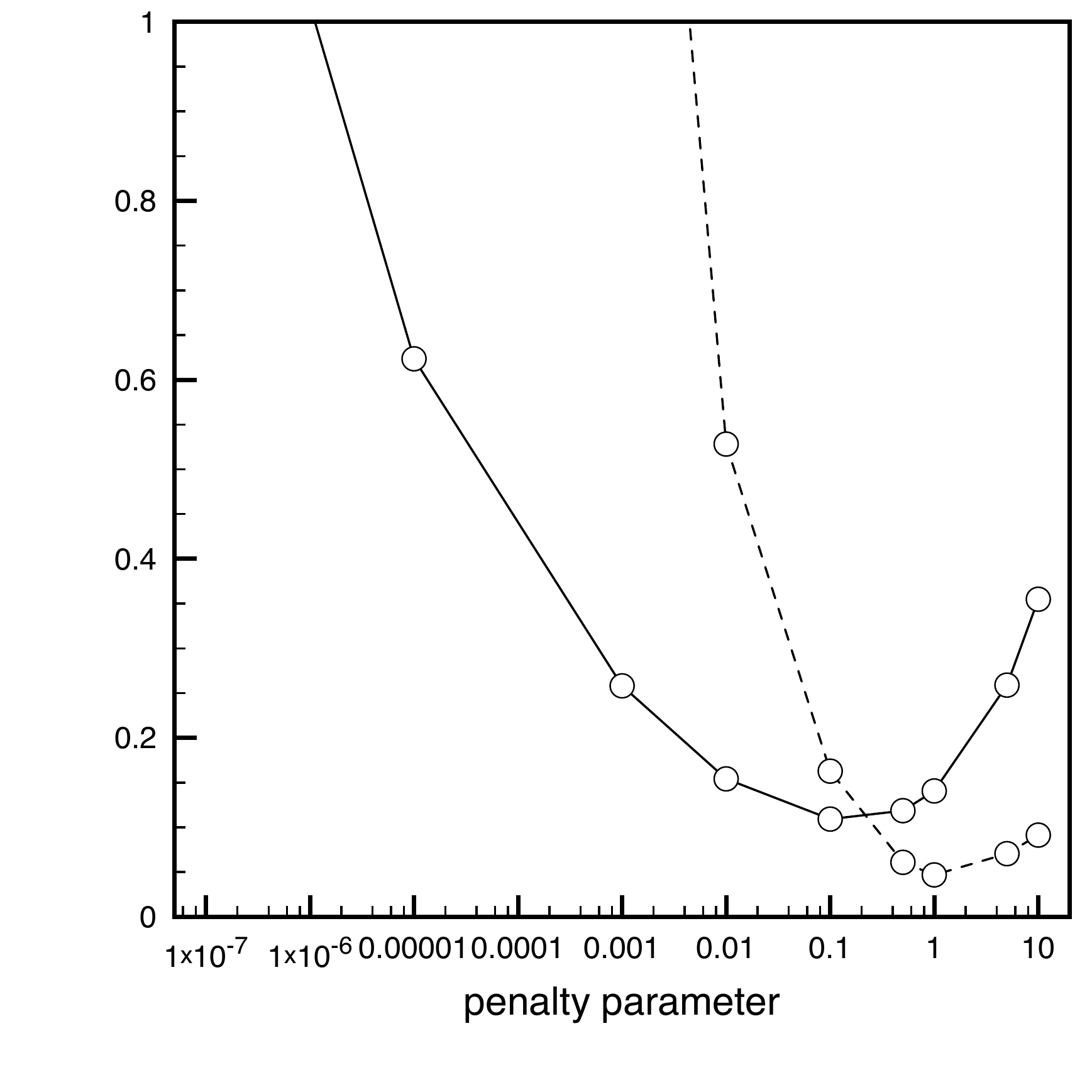}
\caption{Variation of the global $L^2$-error (dashed line) and $|z_h|_{s_W} +
  |u_h|_{s_V}$ (full line) against $\gamma$. Left $k=1$. Right $k=2$.}\label{penalty_pert_study}
\end{figure}
\begin{figure}
\centering
\includegraphics[width=0.47\textwidth]{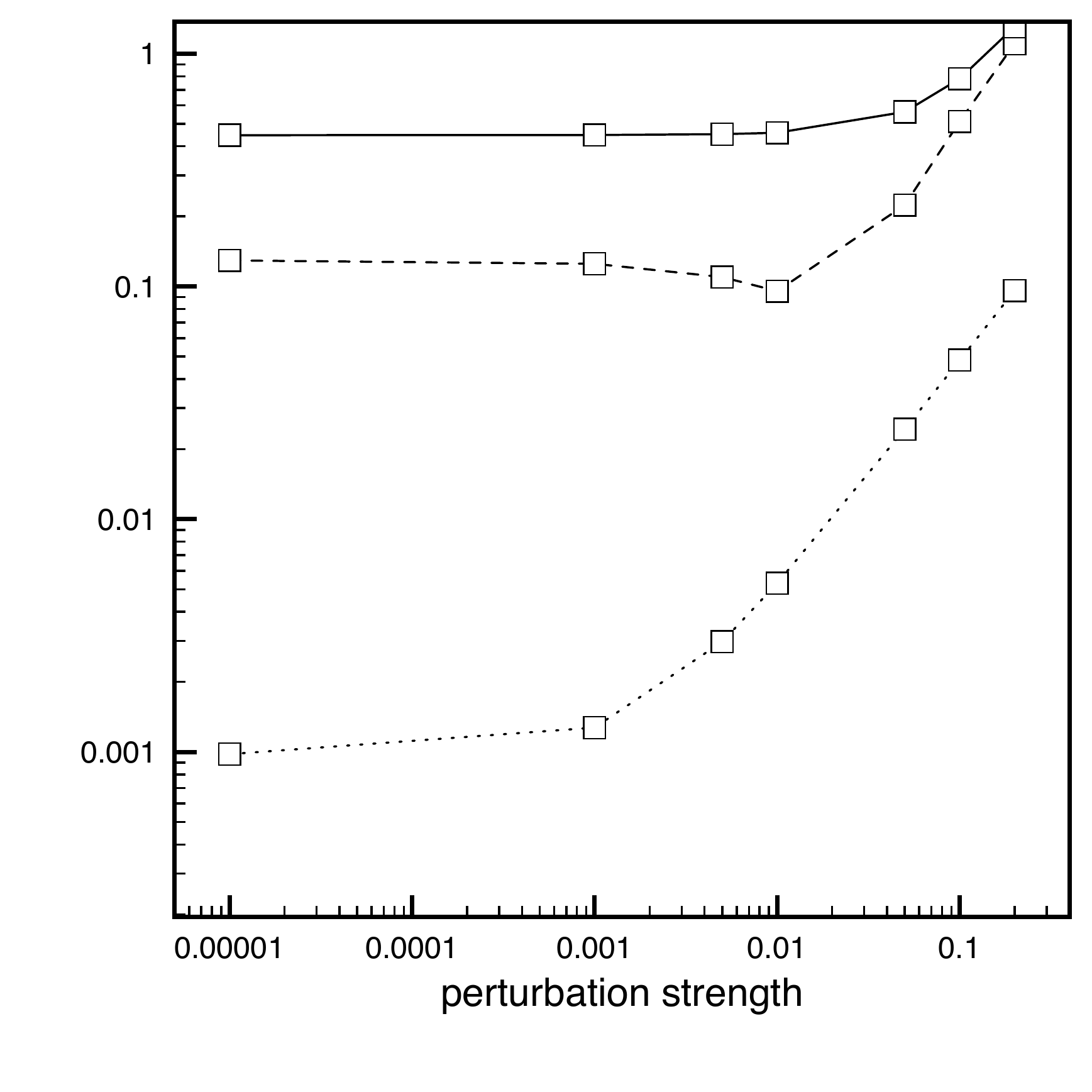}
\includegraphics[width=0.47\textwidth]{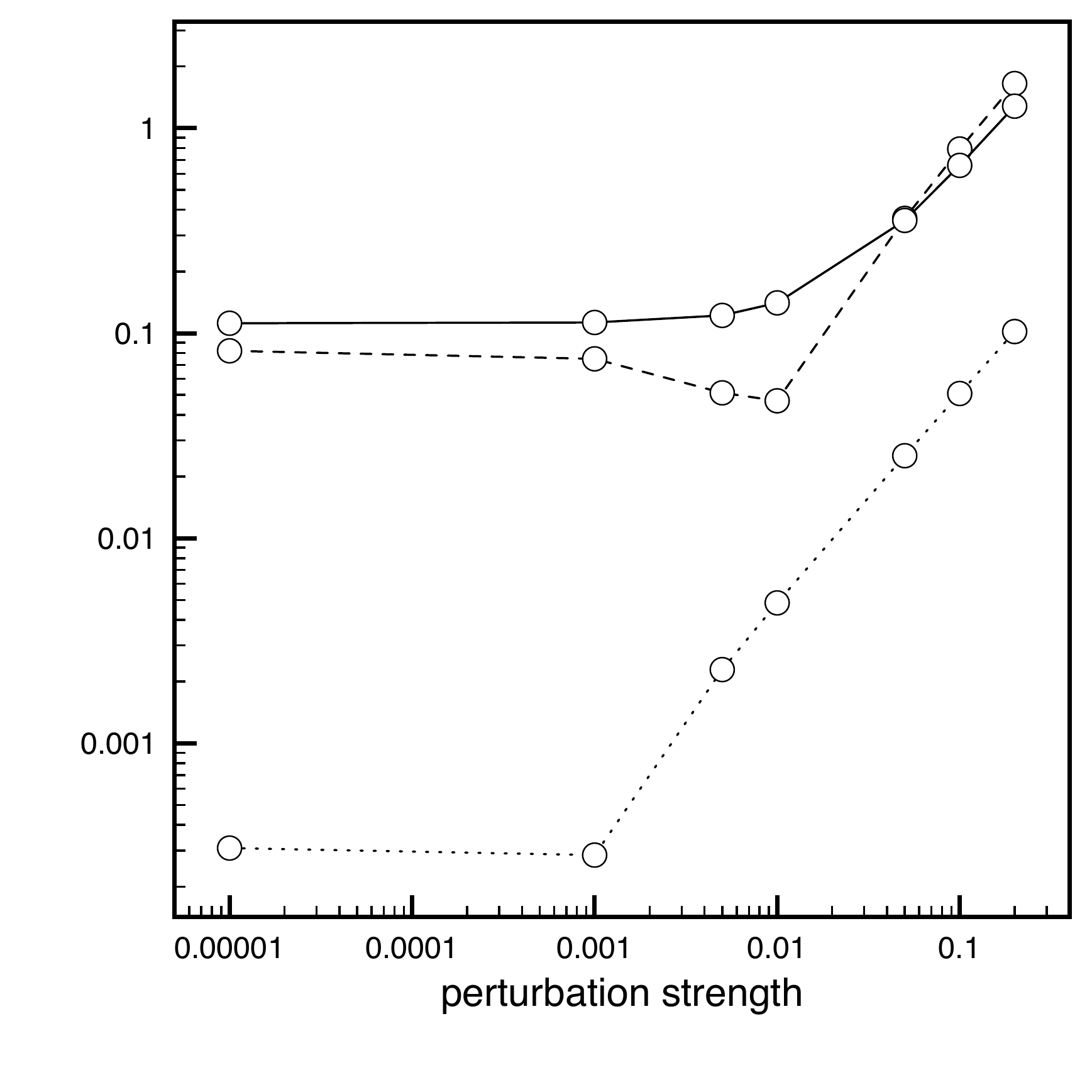}
\caption{Variation of the $L^2$-error (global dashed line, local
  dotted line) and $|z_h|_{s_W} +
  |u_h|_{s_V}$ (full line) against $\varsigma$. Left $k=1$. Right $k=2$.}\label{sens_study}
\end{figure}
\begin{figure}
\centering
\includegraphics[width=0.47\textwidth]{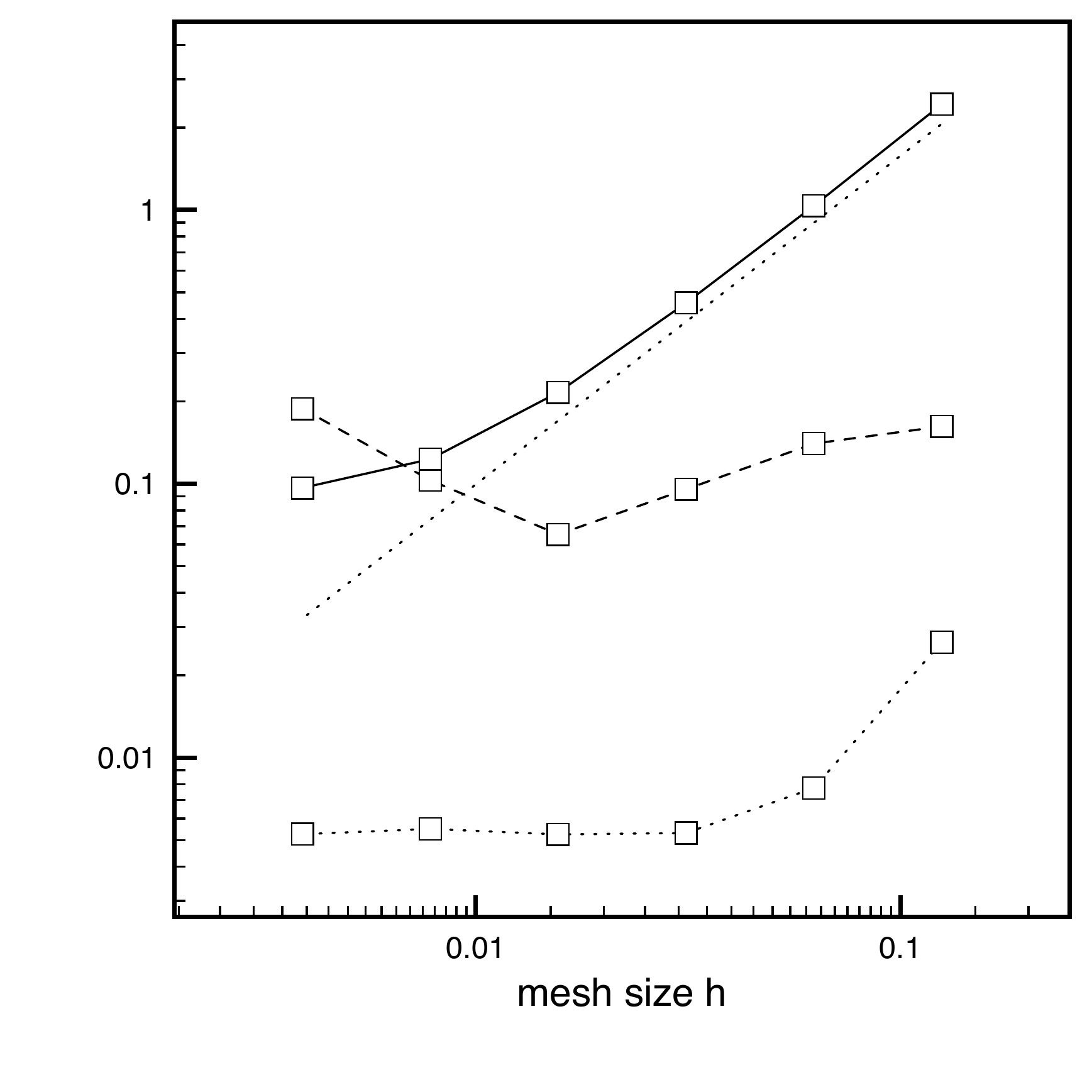}
\includegraphics[width=0.47\textwidth]{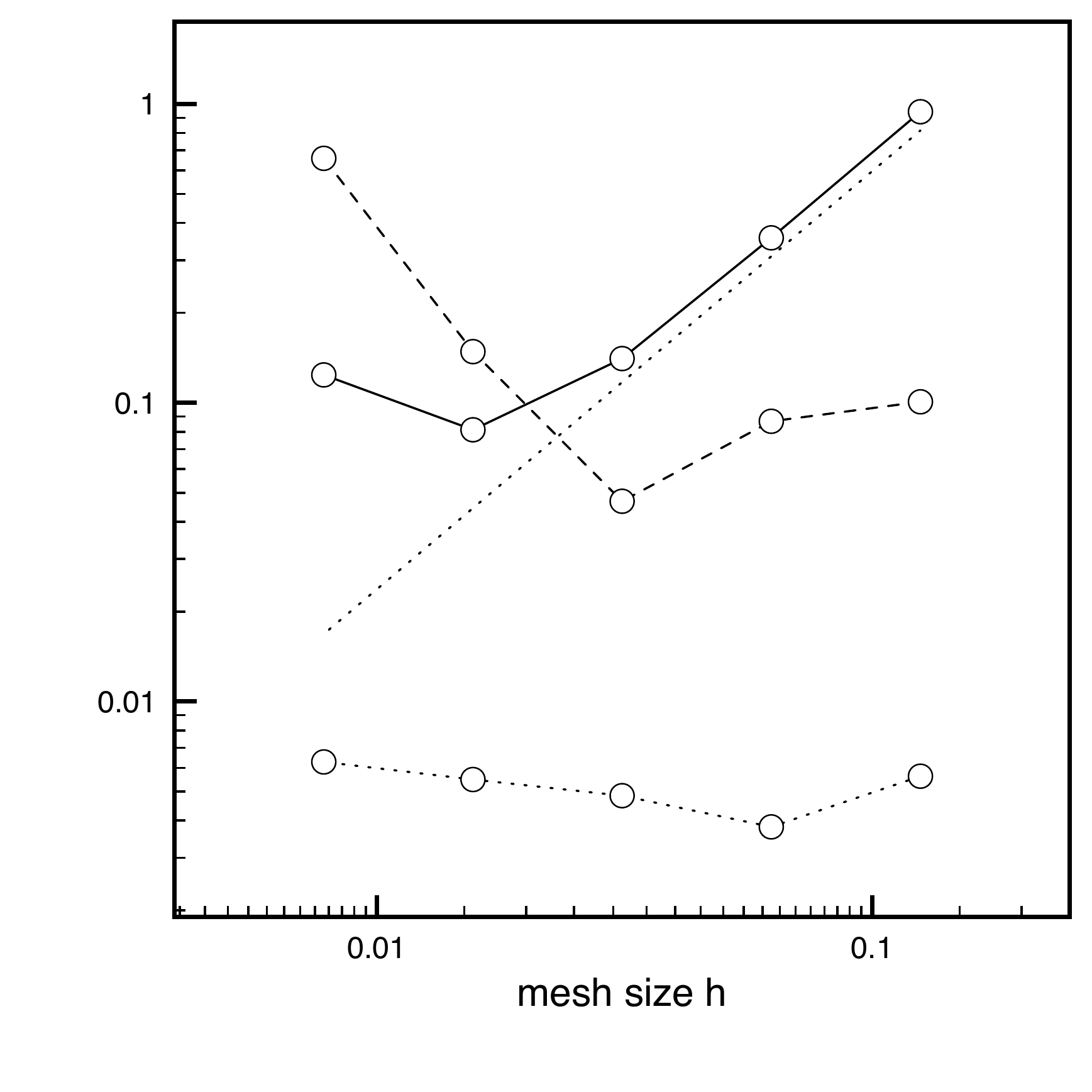}
\caption{Variation of the $L^2$-error (global dashed line, local
  dotted line) and $|z_h|_{s_W} +
  |u_h|_{s_V}$ (full line, with markers) against $h$. Left $k=1$,
  reference $O(h^{1.1})$. Right $k=2$, reference $O(h^{1.4})$}\label{conv_study}
\end{figure}

\subsection{The elliptic Cauchy problem for the convection--diffusion
  operator}
As a last example we consider the Cauchy problem using the
noncoercive convection--diffusion operator \eqref{op:conv_diff}.
The stability of the problem depends strongly on where the boundary conditions
are imposed in relation to the inflow and outflow
boundaries. Strictly speaking this problem is not covered by the
theory developed in \cite{ARRV09}. Indeed in that work the quantitative
unique continuation used the symmetry of the operator. An extension to
the convection-diffusion case is likely to be possible, at least in two space
dimensions, by combining the results of \cite{Ales12} with those of \cite{ARRV09}.

To illustrate the dependence of the stability on how boundary data is
distributed on inflow and outflow boundaries we propose two configurations. Recalling the left 
plot of Figure \ref{velocities_conv} we observe that the flow enters along the boundaries $y=0$, $y=1$ and
$x=1$ and
exits on the boundary $x=0$. Note that the strongest inflow takes
place on $y=0$ and $x=1$, the flow being close to parallel to the
boundary in the right half of the segment $y=1$. We propose the two
different Cauchy problem configurations:
\begin{description}
\item[\bf{Case 1.}] We impose Dirichlet and Neumann data on the two
  inflow boundaries $y=0$ and $x=1$. 
\item[\bf{Case 2.}] We impose  Dirichlet and Neumann data on the two 
  boundaries $x=0$ and $y=1$ comprising both inflow and outflow parts.
\end{description}
The gradient penalty operator has been weighted with the P\'eclet
number as suggested in \cite{Bu13}, to obtain optimal performance in all regimes.
In the first case  the main part of the inflow boundary is included in
$\Gamma_S$ whereas in the second
case the outflow portion {\emph{or}} the inflow portion of every
streamline are included in the boundary portion $\Gamma_S$ where data are set. This highlights
two different difficulties for Cauchy problems for the
convection--diffusion operator, in Case 1 the crosswind diffusion must
reconstruct missing boundary data whereas in Case 2 we must solve the problem
backward along the characteristics, essentially solving a backward
heat equation.

In Figure \ref{Fig:cauchy_conv_diff}, we report the results on the same sequence of unstructured meshes used
in the previous examples for piecewise affine approximations and the two problem configurations.
In the left plot of Figure \ref{Fig:cauchy_conv_diff} we see the
convergence behaviour for Case 1, when piecewise affine approximation
is used. The global $L^2$-norm error clearly reproduces the inverse
logarithmic convergence order predicted by the theory for the
symmetric case. In the right plot
of Figure \ref{Fig:cauchy_conv_diff} we present the convergence plot
for Case 2 (the dotted lines are the same inverse logarithmic
reference curves as in the left plot). In this case we see that the
convergence initially is approximately linear, similarly as that
of the stabilisation term. For finer meshes however the inverse logarithmic
error decay is observed, but with a much smaller constant compared to
Case 1. In Case 1 the diffusion is important on all scales, since some
characteristics have no data neither on inflow or outflow, whereas in
Case 2, data is set either on the inflow or the outflow for all
characteristics of the flow and the effects of diffusion are therefore
much less important, in particular on coarse scales. Indeed the
reduced transport problem in the limit of zero diffusivity, is not ill-posed. As the flow is
resolved the effect of the diffusion once again dominates and the
inverse logarithmic decay reappears.
\begin{figure}
\centering
\includegraphics[width=0.47\textwidth]{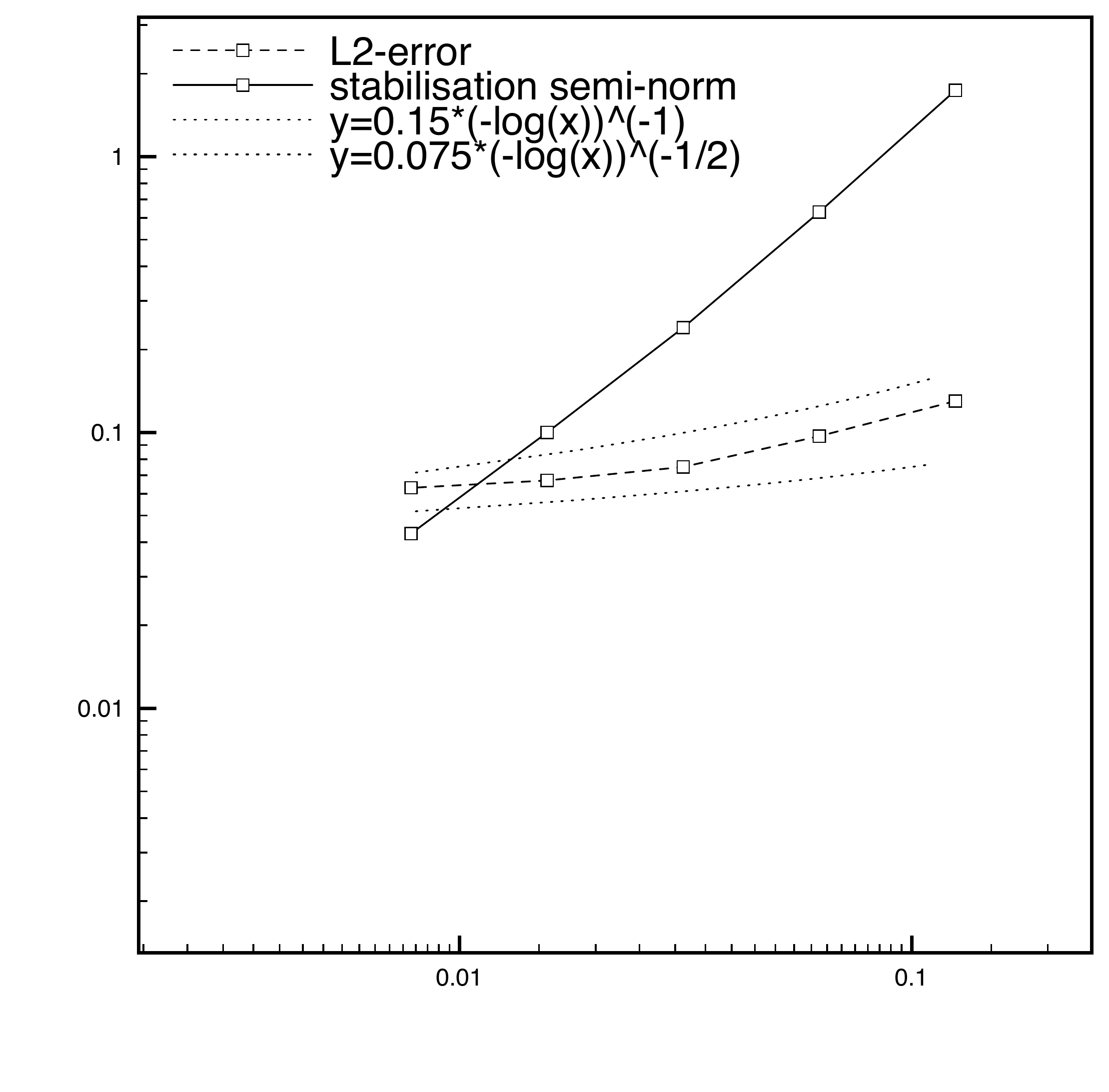}
\includegraphics[width=0.47\textwidth]{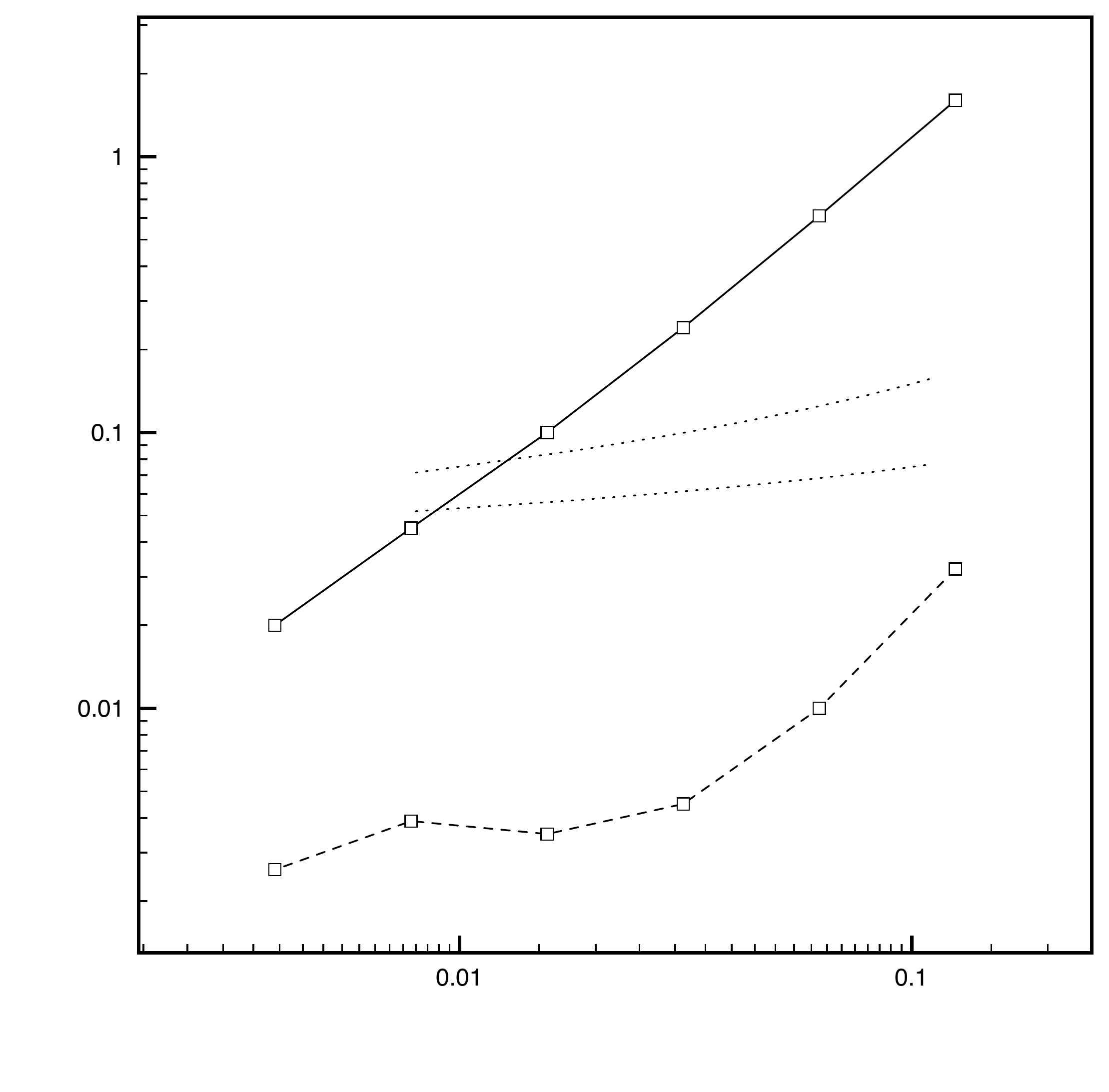}
\caption{Left: convergence for Case 1,  k=1. Right:
  convergence for Case 2, k=1}\label{Fig:cauchy_conv_diff}
\end{figure}
\section{Conclusion}
We have proposed a framework using stabilised finite element methods for the approximation of ill-posed
problems that have a conditional stability property. The key element
is to reformulate the problem as a pde-constrained minimisation
problem that is regularized on the discrete level using tools known
from the theory of stabilised FEM. Using the conditional stability error
estimates are derived that are optimal with respect to the stability of the
problem and the approximation properties of the finite element
spaces. 
The effect of perturbations in data may also be accounted for in the framework and leads to limits on
the possibility to improve accuracy by mesh refinement. Some 
numerical examples were presented illustrating different aspects
of the theory.

There are several open problems both from theoretical and computational
point of view, some of which we will address in future work. Concerning the stabilisation it is not clear
if the primal and adjoint stabilisation operators
should be chosen to be the same, or not?
Does the adjoint consistent choice of stabilisation $s_W$ have any advantages compared
to the adjoint stabilisation \eqref{eq:nonadjointcons_pen}, that gives stronger control of perturbations?
Then comes the question of whether or not high order
approximation (i.e. polynomials of order higher than one) can be
competitive also in the presence of perturbed data? Can the a posteriori error estimate derived
in Theorem \ref{thm:cont_dep} be used to drive adaptive algorithms?
Finally, what is a suitable preconditioner for the linear system? We
hope that the present work will help to stimulate discussion on the
design of numerical methods for ill-posed problems and provide some 
new ideas on how to make a bridge between the regularization methods
traditionally used and (weakly) consistent stabilised finite element
methods.
\section*{Appendix}
We will here give a proof that the inf-sup stability \eqref{disc_stab} holds also
for the stabilisation \eqref{CIP_stab}. We do not track the depedence
on $\gamma_D$ and $\gamma_S$.
\begin{proposition}
Let $A_h[(\cdot,\cdot),(\cdot,\cdot)]$ be defined by \eqref{global_A}
with $a_h(\cdot,\cdot)$, $s_W(\cdot,\cdot)$ and $s_V(\cdot,\cdot)$
defined by equation \eqref{disc_a2}, \eqref{penalty_p2},
\eqref{penalty_a2} and \eqref{CIP_stab} (or \eqref{eq:nonadjointcons_pen} for $s_W(\cdot,\cdot)$). Then the inf-sup condition
\eqref{disc_stab} is satisfied for the semi-norm \eqref{Lnorm}.
\end{proposition}
\begin{proof}
We must prove that the $L^2$-stabilisation of the jump of the
Laplacian gives sufficient control for the inf-sup stability of
$\mathcal{L} u_h$ evaluated elementwise. 
It is well known \cite{BFH06} that for the quasi-interpolation
operator defined in each node $x_i$ by
\[
(I_{os} \Delta u_h)(x_i) := N_i^{-1} \sum_{\{K : x_i \in K\}} \Delta u_h(x_i)\vert_K,
\]
$N_i := \mbox{card} \{K : x_i \in K\}$
the following discrete interpolation result holds 
\begin{equation}\label{oswald_int}
\|h (\Delta u_h - I_{os}
\Delta u_h)\|_h \leq  C_{os} s^S_{V}(u_h,u_h)^{\frac12}
\end{equation}
as well as following the stabilities obtained using trace inequalities, inverse
inequalities and the $L^2$-stability of $I_{os}$,
\begin{equation}\label{Ios_stab}
\|h^{\frac32} I_{os} \Delta u_h\|_{\mathcal{F}} +
\|h^{\frac52} \partial_n I_{os} \Delta u_h\|_{\mathcal{F}} + \|h I_{os} \Delta
u_h\|_h +|h^2 I_{os} \Delta
u_h|_{s_X} \lesssim \|h \Delta u_h\|_h,
\end{equation}
where $X=V,W$.
First observe that by taking $(v_h,w_h)=(u_h,z_h)$ we have
\[
|u_h|_{s_V}^2 + |z_h|_{s_W}^2 = A_h[(u_h,z_h),(u_h,z_h)].
\]
Now let $w^{\mathcal{L}}_h=h^2 I_{os} \mathcal{L} u_h  = h^2 (I_{os} \Delta u_h + c
u_h)$, $v^{\mathcal{L}}_h=h^2 I_{os} \mathcal{L}^* z_h$. Using \eqref{Ios_stab} it
is straightforward to show that
\begin{equation}\label{Ios_stab2}
\|h^{\frac32} I_{os} \mathcal{L} u_h\|_{\mathcal{F}} +
\|h^{\frac52} \partial_n I_{os} \mathcal{L} u_h\|_{\mathcal{F}} + \|h  I_{os} \mathcal{L} u_h\|_h +|h^2 I_{os} \mathcal{L} u_h|_{s_X} \leq \tilde C_{os} \|h \mathcal{L}  u_h\|_h.
\end{equation}
Now observe that (for a suitably chosen orientation of the normal on
interior faces)
\begin{multline*}
a_h(u_h,w^{\mathcal{L}}_h) = \|h \mathcal{L} u_h\|^2_h +  (\mathcal{L} u_h, h^2
(I_{os} \mathcal{L} u_h - \mathcal{L} u_h))_h +  \left<\jump{\partial_n u_h}, h^2
I_{os} \mathcal{L} u_h \right>_{\mathcal{F}_I}\\
 + \left<\partial_n u_h, h^2
I_{os} \mathcal{L} u_h \right>_{\Gamma_N}  + \left<\partial_n h^2
I_{os} \mathcal{L} u_h, u_h \right>_{\Gamma_D} \\
\ge \frac12 \|h
\mathcal{L} u_h\|^2_h -  2 \|h^2 (I_{os} \mathcal{L} u_h -
\mathcal{L} u_h)\|_h^2- 2 \tilde C_{os}^{-2}  s_V^D(u_h,u_h)\\
\ge  \frac12 \|h
\mathcal{L} u_h\|^2_h - 2 C^2_{os}  s_V^S(u_h,u_h)-
2 \tilde C_{os}^{-2}  s_V^D(u_h,u_h) \\
\ge \frac12 \|h
\mathcal{L} u_h\|^2_h - 2(C^2_{os} +  \tilde C_{os}^{-2}) |u_h|_{s_V}^2
\end{multline*}
and
\[
s_W(z_h,w^{\mathcal{L}}_h) \ge -\tilde C_{os}^{-2} |z_h|_{s_W}^2 -
\frac14 \|h \mathcal{L}  u_h\|_h^2.
\]
Similarly
\[
a_h(v^{\mathcal{L}}_h,z_h) \ge \frac12 \|h
\mathcal{L}^* z_h\|^2_h - 2(C^2_{os} +  \tilde C_{os}^{-2}) |z_h|_{s_W}^2
\]
and
\[
s_V(u_h,v^{\mathcal{L}}_h) \ge - \tilde C_{os}^{-2} |u_h|_{s_V}^2 - \frac14\|h \mathcal{L}^*  z_h\|_h^2.
\]
It follows that for some $c_1,c_2>0$ there holds
\[
|(u_h,z_h)|_{\mathcal{L}}^2 \lesssim A_h[(u_h,z_h),(u_h+ c_1 w^{\mathcal{L}}_h ,z_h+ c_2
v^{\mathcal{L}}_h)].
\]
We conclude by observing that by inverse inequalities and \eqref{Ios_stab2} we have the
stability
\[
|(u_h+ c_1 w^{\mathcal{L}}_h,z_h+ c_2
v^{\mathcal{L}}_h)|_{\mathcal{L}} \lesssim |(u_h,z_h)|_{\mathcal{L}}.
\]
\qed
\end{proof}
\bibliographystyle{abbrv}
\bibliography{Cauchy}

\end{document}